\documentclass[12pt]{amsart}
\usepackage[top=3.0cm, bottom=3.0cm, left=2.5cm, right=2.5cm]{geometry}
\usepackage{fullpage}
\newtheorem{theorem}{Theorem}[section]
\usepackage{framed}
\usepackage{xcolor}
\definecolor{refcyan}{rgb}{0,0,0}
\usepackage{tikz}
\newtheorem{lemma}[theorem]{Lemma}
\newtheorem{proposition}[theorem]{Proposition}
\newtheorem{corollary}[theorem]{Corollary}
\usepackage{amsmath}
\usepackage{graphicx} 
\usepackage{float}
\usepackage{amsfonts}
\usepackage{amssymb}
\newtheorem{definition}[theorem]{Definition}
\newtheorem{example}[theorem]{Example}

\newtheorem{remark}[theorem]{Remark}
\numberwithin{equation}{section}
\usepackage{hyperref,xcolor}
\hypersetup{colorlinks=true, citecolor=black, linkcolor=black, urlcolor=black}

\usepackage{etoolbox}
\usepackage[font=small,labelfont=bf]{caption}
\usepackage{titletoc}
\begin{document}
\author[Beslikas]{Athanasios Beslikas}
\address{Doctoral School of Exact and Natural Studies,
Institute of Mathematics,
Faculty of Mathematics and Computer Science,
Jagiellonian University,
\L{}ojasiewicza 6, PL30348, Cracow, Poland}
\email{athanasios.beslikas@doctoral.uj.edu.pl}

\title{Composition operators and Rational Inner Functions on the bidisc: A geometric approach}

\begin{abstract} We study composition operators acting on 
the weighted Bergman spaces on the bidisc, i.e.
$C_{\Phi}:A^2_{\beta}(\mathbb{D}^2)\to A^2_{\beta}(\mathbb{D}^2)$ where $\Phi$ is induced by rational inner functions (RIFs) or a RIF and a smooth function (mixed case). Our approach is geometric. Our main result is a uniform criterion for all $\beta\in(-1,0]$ that can be summarized as follows:
Boundedness of the composition operator is equivalent to transversal intersection of the level sets for non-smooth symbols, under the assumption that if any tangential intersection occurs on the singularity it must be of high order.
This extends the characterization of Bayart-Kosi\'nski to the non-smooth self maps of the bidisc. To reach our conclusions, we utilize results obtained by Anderson, Bergqvist, Bickel, Cima and Sola on Clark measures associated to RIFs and Puiseux factorizations. 
\end{abstract}

\maketitle
\section{Introduction} Let $$\mathbb{D}^2:=\bigg\{z=(z_1,z_2)\in\mathbb{C}^2:|z_1|<1,|z_2|<1\bigg\}$$ denote the unit bidisc in two complex variables and $$\mathbb{T}^2:=\bigg\{\zeta=(\zeta_1,\zeta_2)\in\mathbb{C}^2:|\zeta_1|=1,|\zeta_2|=1\bigg\}$$ denote the bitorus, the \textit{distinguished boundary} of the bidisc. 
Consider $f\in \mathcal{O}(\mathbb{D}^2,\mathbb{C}),$ a holomorphic function on the bidisc,  and $\Phi\in\mathcal{O}(\mathbb{D}^2,\mathbb{D}^2)$ a holomorphic self map of the said domain. The composition operator
$$C_{\Phi}(f):=(f\circ\Phi)(z_1,z_2),\quad (z_1,z_2)\in\mathbb{D}^2,$$
defines a new holomorphic function on the bidisc.

Let us for a moment be more abstract. The first problem that one studies from a functional-analytic point of view is when this operator defines a bounded linear operator on a Banach or Hilbert space $X,$ i.e. for what self maps $\varphi$ of the domain $\Omega\subset\mathbb{C}^n,$ the inequality

$$\|C_{\varphi}(f)\|_{X}\leq C\|f\|_{X},$$
is satisfied.
In the case where $X=H^2(\mathbb{D)}$ or $X=A^2_{\beta}(\mathbb{D})$ the Hardy and (weighted) Bergman spaces of the unit disc, the problem is resolved via the celebrated \textit{Littlewood subordination principle}; all self maps $\varphi\in\mathcal{O}(\mathbb{D},\mathbb{D})$ induce bounded composition operators in both spaces.

 Inspired by the work of Shapiro (see the book \cite{Shapiro} for a comprehensive study of composition operators in one complex variable), who built the geometric foundation in the study of composition operators on the one-dimensional setting, this work pursues a geometric interpretation of the problem in two complex variables. In one variable, understanding the geometry of the self map $\varphi\in\mathcal{O}(\mathbb{D},\mathbb{D}),$ leads to results on compactness of the composition operator in Hardy and Bergman spaces. As an example that highlights this geometrical aspect of the problem in one complex variable, Shapiro states in his seminal book \cite{Shapiro}, page 33, the following intuitive description of compactness of the composition operator;

\begin{center} \textit{`` $C_{\varphi}$ is compact \\if and only if $\varphi(z)$ does not get too close to the unit circle too often."}
\end{center}

More recently, mathematicians have turned their attention to the polydisc $\mathbb{D}^n$ and the action of the composition operator on Banach and Hilbert spaces defined on this domain. The problem of characterizing the symbols $\Phi$ that induce bounded composition operator in the Hardy and Bergman spaces of the bidisc was first considered in the work of Koo, Stessin and Zhu \cite{Koo}. Later, Bayart partially corrected the previous results and his paper \cite{Bayart1} together with the more recent paper of Kosi\'nski  \cite{CompositionsKosinski}, led to a characterization of $\mathcal{C}^2(\overline{\mathbb{D}}^2)-$smooth self maps that induce bounded composition operators on the Hardy and Bergman space of the bidisc. This characterization is often reffered to as the \textit{first-order condition}.

All of the studies are technical in nature, a testament of the difficulty of the problem in several complex variables. The higher dimensional setting restricts our intuition on the geometry of the self maps of the bidisc and how these mappings behave when their values approach the  bitorus. In this paper, our goal is to study the self maps of the bidisc with singular behavior, appearing in the papers \cite{Me, Me2}, but this time we focus on the geometry of the self maps. Specifically, for each coordinate function we will picture its level sets on $\mathbb{T}^2,$ compare them, and collect results based on this comparison. Our goal is to relate our results to the \textit{first-order condition.}

For convenience of the reader, let us state this result.

\begin{theorem} (First-order condition) Let $\Phi \in \mathcal{O}(\mathbb{D}^2,\mathbb{D}^2)\cap \mathcal{C}^2(\overline{\mathbb{D}^2})$. Then the composition operator $C_{\Phi}:A^2_{\beta}(\mathbb{D}^2)\to A^2_{\beta}(\mathbb{D}^2) $ is bounded, if and only if for all $\zeta\in \mathbb{T}^2$ such that $\Phi(\zeta)\in \mathbb{T}^2,$ the derivative $d_{\zeta} \Phi$ is invertible. 
\label{characterization}
\end{theorem}
One could reformulate the first-order condition geometrically as follows (see Proposition \ref{transverse} for a proof of this fact).
\begin{proposition}\label{transverse}\textbf{(First-order condition is equivalent to transversality)}
Let $\Phi=(\phi,\psi)\in \mathcal O(\mathbb D^2,\mathbb D^2)\cap \mathcal{C}^2(\overline{\mathbb D^2}),$
and consider all $\zeta\in \mathbb T^2$ such that
$\Phi(\zeta)\in \mathbb T^2.$
{\color{refcyan}Assume that, for every such $\zeta$, the level sets
$\mathcal{C}_{\phi(\zeta)}(\phi)$ and $\mathcal{C}_{\psi(\zeta)}(\psi)$
are smooth curves locally at $\zeta$.}
Then the following are equivalent:
\begin{enumerate}
    \item The derivative $d_{\zeta}\Phi$ is invertible for all $\zeta\in\mathbb{T}^2$ such that $\Phi(\zeta)\in\mathbb{T}^2$.
    \item The boundary level curves
    $\mathcal{C}_{\phi(\zeta)}(\phi)\quad \mathrm{and}\quad
    \mathcal{C}_{\psi(\zeta)}(\psi)$
    are smooth for all $\zeta\in\mathbb{T}^2$ such that $\Phi(\zeta)\in\mathbb{T}^2$ and intersect transversely there.
    \item $C_{\Phi}:A^2_{\beta}(\mathbb{D}^2)\to A^2_{\beta}(\mathbb{D}^2)$ is bounded.
\end{enumerate}
\end{proposition}
Here, by \textit{transversality} we mean that the curves intersect at a strictly positive angle.
The target of the present study is to prove a geometric principle for boundedness of the composition operator in the spirit of the compactness criterion of Shapiro, that is 

\begin{center}\textit{ `` The composition operator $C_{\Phi}$ is bounded on the Bergman space of the bidisc,\\ if and only if the coordinate functions touch the bitorus transversally".}
\end{center}
{\color{refcyan}This, as we saw, is true for smooth symbols under the assumptions of Proposition \ref{transverse}, but does this result extend to non-smooth cases?}

\subsection{Structure of the paper} 
 In the next two subsections \ref{Bergman} and \ref{RIFs} of the introduction, we provide the setting in which we will study our problem. In Section \ref{results}, we state explicitly our main results that contribute in the study of composition operators. In Section \ref{tools}, we provide the necessary lemmata-theorems-propositions that we shall invoke in the duration of the paper. We split this section in 3 smaller ones for convenience of the reader. Then, we provide the proofs of the main results in Section \ref{proofs}. In Section \ref{examples} we present four concrete examples that illustrate several of our main results in detail.  

\subsection{Bergman spaces on the bidisc and Carleson measures}\label{Bergman}
Among the most natural spaces of holomorphic functions are the (weighted) Bergman spaces in a domain. Our focus will be the polydisc $\mathbb{D}^n,$ and particularly the bidisc $\mathbb{D}^2$. For $\beta\ge-1,$ the \textit{weighted Bergman spaces} are 
$$A^2_{\beta}(\mathbb{D}^2):=\Bigg\{f\in \mathcal{O}(\mathbb{D}^2):\int_{\mathbb{D}^2}|f(z_1,z_2)|^2(1-|z_1|^2)^{\beta}(1-|z_2|^2)^{\beta}dV(z_1,z_2)<+\infty\Bigg\}.$$ Here, $dV(z_1,z_2)=dA(z_1)\times dA(z_2),$ where $dA(z)=\frac{1}{\pi}dxdy,\quad z=x+iy\in\mathbb{D},$ denotes the \textit{normalized Lebesgue measure} on the unit disc. For the sake of notational clarity, $dV_{\beta}$ will denote the radially weighted volume measure.

Whenever $\beta=-1$ we recover the Hardy spaces $A^2_{-1}(\mathbb{D}^2)=H^2(\mathbb{D}^2),$ and for $\beta=0$ we obtain the classical Bergman space $A^2_0(\mathbb{D}^2)=A^2(\mathbb{D}^2).$ The main tool in the paper is the characterization of \textit{Carleson measures} for the weighted Bergman spaces. Carleson measures for the radially weighted Bergman spaces have been fully characterized, thus providing us with a sufficient and necessary condition for the boundedness of the composition operator. Before we state the related Lemma, let us recall some definitions. Let $\tilde{\delta}=(\delta_1,\delta_2)\in (0,1)^2$ and $\zeta=(\zeta_1,\zeta_2)\in\mathbb{T}^2.$ The \textit{two-dimensional Carleson box} is the set $$S(\zeta,\tilde{\delta})=\{z\in\mathbb{D}^2:|z_j-\zeta_j|<\delta_j, \quad j=1,2\}.$$ The characterization of the Carleson measures that can be found in \cite{carleson} gives us for free a straightforward criterion for the boundedness of the composition operator $C_{\Phi}$ on $A^2_{\beta}(\mathbb{D}^2),$ that is
\begin{lemma}
Let $\Phi:\mathbb{D}^2\to \mathbb{D}^2$ be a holomorphic self-map of the bidisc. The composition operator $C_{\Phi}: A^2_{\beta}(\mathbb{D}^2)\to A^2_{\beta}(\mathbb{D}^2)$ is bounded if and only if there is a constant $C>0$ such that for every $\tilde{\delta} \in (0,2)^2$ and $\zeta\in\mathbb{T}^2$:
$$V_{\beta}(\Phi^{-1}(S(\zeta,\tilde{\delta})))\leq CV_{\beta}(S(\zeta,\tilde{\delta})).$$
\end{lemma}
It is well understood that the volume of the Carleson box in two dimensions behaves like $V_{\beta}(S(\zeta,\delta))\approx\delta_1^{\beta+2}\delta_2^{\beta+2}.$

In the later sections of our paper, we shall need the corresponding non-boundedness criterion of this Theorem. Specifically, if for some $\zeta\in\mathbb{T}^2,$ one shows that
$$\lim_{\delta\to0}\frac{V_{\beta}(\Phi^{-1}(S(\zeta,\delta)))}{\delta^{2\beta+4}}=+\infty,$$
then non-boundedness of the composition operator $C_{\Phi}: A^2_{\beta}(\mathbb{D}^2)\to A^2_{\beta}(\mathbb{D}^2)$ is implied. {\color{refcyan}In the non-boundedness arguments below, the notation $S(\zeta,\delta)$ stands for a box with equal side lengths $\delta_1=\delta_2=\delta$. It is enough to test this particular family to disprove the Carleson estimate.}

\subsection{Rational Inner Functions on the bidisc}\label{RIFs}
In complex analysis and holomorphic function spaces, the concept of \textit{inner functions} finds its source in the theory Hardy space $H^2$ on the unit disc $\mathbb{D},$ specifically, in the inner-outer factorization Theorem of Riesz. A function $\phi:\mathbb{D}\to\mathbb{D}$ is said to be \textit{inner}, if the non-tangential limit function $\phi^{*}$ is unimodular almost everywhere on the circle $\mathbb{T}.$ {\color{refcyan}A function $\phi$, defined on the bidisc, is called} \textit{inner} if the non-tangential limit function $\phi^*$ is unimodular almost everywhere on $\mathbb{T}^2,$ a direct counterpart of the one dimensional definition. Non-tangential approach regions on the bidisc or polydisc are essentially products of non-tangential approach regions on the disc. 

The rational functions that are inner on the unit disc $\mathbb{D}$ are the finite Blaschke products, i.e. $$\phi(z_1)=e^{i\theta}z_1^N\prod_{j=1}^N\frac{a_j}{|a_j|}\frac{z_1-a_j}{1-\overline{a_j}z_1}, \quad z_1\in\mathbb{D}, \quad a_j\in \mathbb{D},\quad j=1,...,N.$$

In several complex variables, Rudin, Stout and Pfister (see \cite{Rudin,Pfister}), proved a characterization of the Rational Inner Functions (in brief RIFs). To be able to state this result, we need the notion of \textit{stable polynomials} $p\in\mathbb{C}[z_1,z_2]$ on the bidisc. \textit{Stable} in our context means that the polynomial does not vanish on $\mathbb{D}^2.$ {\color{refcyan}In the most singular case, the polynomial has zeros on the distinguished boundary $\mathbb{T}^2.$} Take such a polynomial $p(z_1,z_2)$ with bidegree $\deg(p(z_1,z_2))=(n,m)\in\mathbb{N}^2$ and define its reflection
$$\widetilde{p}(z_1,z_2):=z_1^nz_2^m\overline{p\left(\frac{1}{\overline{z_1}},\frac{1}{\overline{z_2}}\right)},\quad (z_1,z_2)\in\mathbb{D}^2.$$
Then, all Rational Inner Functions admit the representation

$$\phi(z_1,z_2)=\lambda z_1^Nz_2^M\frac{\widetilde{p}(z_1,z_2)}{p(z_1,z_2)}\quad (z_1,z_2)\in\mathbb{D}^2,\quad N,M\in\mathbb{N}.$$

Two prominent examples of RIFs on the bidisc are $\kappa(z_1,z_2)$ and the AMY (Agler-McCarthy-Young) functions, namely.

$$\kappa(z_1,z_2)=\frac{2z_1z_2-z_1-z_2}{2-z_1-z_2} \quad \mathrm{and}\quad \phi_{AMY}(z_1,z_2)=\frac{4z_1^2z_2-z_1^2-3z_1z_2-z_1+z_2}{4-3z_1-z_2-z_1z_2+z_1^2}.$$

Since the geometric results that we will present in the sequel depend only on the level-set geometry near boundary singularities, the monomial factors play no essential role in our discussion. We restrict our attention to functions of the form 
$\phi=\widetilde{p}/p.$ 

A significant distinction between finite Blaschke products and RIFs in several complex variables is that the former extend holomorphically to a disc strictly larger than $\mathbb{D},$ whereas RIFs on the bidisc (and polydisc in general) have \textit{singularities} on the distinguished boundary $\mathbb{T}^2.$ The nature of these singularities has been studied extensively by Knese, Bickel, Pascoe, Sola among others, in the papers \cite{PLMS,Pisa,Clark}. These singularities might prevent holomorphic extensions on $\mathbb{T}^2$, but RIFs still exhibit a rather nice non-tangential behavior. A fundamental Theorem of Knese assures us that RIFs have {\color{refcyan}non-tangential limits \textit{everywhere}} on $\mathbb{T}^2$ and, moreover, they might possess even higher non-tangential regularity, admitting non-tangential-Taylor expansions. Specifically, characterizations of non-tangentially smooth RIFs can be found in the seminal work of Knese \cite{Knese2}. The reader can also consult the papers \cite{Polonici,AJM,Hilbert1, Knese1, Knese2, TAMSKNESE, Sola} for more information and background on RIFs.

Such functions feature in problems in complex analysis, operator theory and also, they share connections with geometry, algebraic geometry, dynamics, even with engineering. To be more precise, rational inner functions are inducing solutions of interpolation problems on the polydisc (see Agler's interpolation theorem \cite{AMY, AglerInterpolation}, Kosi\'nski's work on the three point Nevanlinna-Pick interpolation problem \cite{Lukasz} among others). They exhibit interesting geometric phenomena on their level set $$\mathcal{C}_{\alpha}(\phi)=\mathrm{clos}\{\zeta\in\mathbb{T}^2:\phi^*(\zeta)=\alpha\}=\{\zeta\in\mathbb{T}^2:{\color{refcyan}\widetilde{p}(\zeta)=\alpha p(\zeta)}\},$$ where $\alpha\in\mathbb{T}.$ 
We underline here that the level sets of RIFs will play a vital role in the proofs of the majority of our main results (see next section).

More generally, one could say that RIFs are intimately connected with classic algebraic geometry theory at the level of the book of Fulton (see \cite{Fulton}), specifically with factorization Theorems. Moreover, they appear in Hardy and Dirichlet space theory on the polydisc. The membership problem of RIFs on the Dirichlet-type spaces and their partial derivatives on the Hardy space of the bidisc, has attracted a significant amount of attention and has produced many important sharp results and characterizations, see for instance \cite{PLMS, MeSola} as well as open conjectures (see conjecture 1 in \cite{MeSola}). Last but not least, RIFs possess the so-called \textit{transfer function realization}, a representation of them defined via matrices, that has applications in engineering, see \cite{Kumert} for more comprehensive details.

\subsection{Notational conventions} We write $A\approx B$ if there exist positive constants $C_1,C_2$ such that $C_1B\leq A\leq C_2B.$ If one side inequality holds, i.e. $A\leq C_1B$ then we write $A\lesssim B.$ To avoid introducing too many constants, $C>0$ will denote a positive constant independent of any other parameter that might occur.

\section{Statement of results}\label{results}  
In this paper we consider two categories of self maps of the bidisc. The one category consists of self maps that have one singular RIF and one smooth function in its coordinates. The second category is RIF symbols. Our main result is
Theorem \ref{characterization1} in which we characterize the boundedness of mixed symbols induced by singular RIFs and a smooth function up to the closure of the bidisc. Essentially, we show that transversal intersection of the level sets of the RIF and the smooth function is equivalent to boundedness. To reach  our main result, we prove first a series of geometric conditions that lead to non-boundedness of the operator. We do this for RIF self maps but these conditions can easily be extended to the mixed case. To reach to our conclusions, we study the level sets of two RIFs and how these interact. We prove
\begin{itemize}
\item \textit{Theorem \ref{exceptionalline} (common exceptional lines, for singular RIFs).}
\item \textit{Theorem \ref{exceptionalcurve} (common (arc of an) analytic curve for both mixed and RIF symbols)}, 
\item \textit{Theorem \ref{smoothtangentlines} (tangential intersection on smooth points for both RIF and mixed symbols) }.
\item \textit{Theorem \ref{singlulartangentlines} (tangential intersection on singularity for RIF symbols with common singularity)} 
\end{itemize}
\begin{remark} All of the results stated above can be easily modified for symbols that are mixed, i.e. one coordinate is a singular RIF and the other a $\mathcal{C}^2-$smooth function up to the closure of the bidisc.
\end{remark}
This last result will pose a significant obstacle; no criterion can exist that characterizes boundedness of composition operators with symbols induced by RIFs with common singularity (or mixed symbols), which is uniform for all $\beta>-1$ whenever the level sets are tangential to the singularity with order $M.$ The answer should depend quantitatively on the \textit{contact order} $K_{\phi}$ of each RIF, a geometric quantity that we shall define in Section \ref{tools}.

\begin{theorem}\label{characterization1}\textbf{(Transversality is equivalent to boundedness)} Let $\Phi=(\phi,\psi)$ with $\phi=\widetilde{p}/p$ a RIF with a singularity at $\tau\in\mathbb{T}^2,$ and $\alpha_1\in\mathbb{T}$ its generic value at $\tau\in\mathbb{T}^2$. Let $\psi\in\mathcal{O}(\mathbb{D}^2,\mathbb{D})\cap \mathcal{C}^2(\overline{\mathbb{D}^2})$ with $\psi(\tau)=\alpha_2\in\mathbb{T}.$ Assume that if the level sets $\mathcal{C}_{\alpha_1}(\phi)$ and $\mathcal{C}_{\alpha_2}(\psi)$ intersect tangentially at $\tau\in\mathbb{T}^2$, then the order of tangential intersection $M$ satisfies $M>2K+1$ where $K=\max\{K_{\phi},K_{\psi}\}.$  Then, the composition operator $C_{\Phi}:A^2_{\beta}(\mathbb{D}^2)\to A^2_{\beta}(\mathbb{D}^2)$ is bounded for all $\beta\in(-1,0],$ if and only if $\mathcal{C}_{\lambda_1}(\phi)$ and $\mathcal{C}_{\lambda_2}(\psi)$ intersect transversally for all $\lambda_1,\lambda_2\in \mathbb{T}.$
\end{theorem}
Here, $K_{\psi}$ is the order of tangency of the level curve of $\psi$ at $\omega\in\mathbb{T}^2.$ {\color{refcyan}The hypothesis concerns the level set of $\phi$ corresponding to the chosen generic value $\alpha_1$ and the level set of $\psi$ corresponding to its boundary value $\alpha_2=\psi(\tau)$. The conclusion concerns all pairs of unimodular values whose level sets meet.}
For the proof of this result, the necessity direction of the proof requires all of the geometric results stated below. Note that we state them for RIF symbols, but the proofs can be modified in the half smooth-half singular case of self-maps. 
\begin{theorem}\label{exceptionalline}\textbf{(Common exceptional line)} Let $\phi_1,\phi_2$ be RIFs. If there exist $\alpha_1,\alpha_2\in\mathbb{T}$ such that both level sets $\mathcal{C}_{\alpha_1}(\phi_1)$ and $\mathcal{C}_{\alpha_2}(\phi_2)$ contain the same vertical line $\{\zeta\in \mathbb{T}^2:\zeta_1=\tau\}$ (and analogously, horizontal), then the composition operator $C_{\Phi}:A^2_{\beta}(\mathbb{D}^2)\to A^2_{\beta}(\mathbb{D}^2)$ with symbol $\Phi=(\phi_1,\phi_2)$ is not bounded, for all $\beta\ge-1$.
\end{theorem}  

The situation in the next main result can be explained briefly as follows. Assume that both of the level sets of two different RIFs contain at least one common (arc of an) analytic curve. Then the composition operator induced by the self map $\Phi=(\phi_1,\phi_2)$ of the bidisc, is not bounded on $A^2_{\beta}(\mathbb{D}^2).$
\begin{theorem}\label{exceptionalcurve}\textbf{(Common analytic curve)} Let $\Phi=(\phi_1,\phi_2)$ induced by RIFs $\phi_\ell=\frac{\widetilde{p}_\ell}{p_\ell}$, where $\ell=1,2.$
Assume that there exists  $\alpha_1,\alpha_2\in\mathbb{T}$ and a closed arc $J\subset \mathbb{T}$ of positive length such that $g^{\alpha_1}_{j_1}(\zeta)=g^{\alpha_2}_{j_2}(\zeta)$, for $\zeta\in J$ and some indices $j_1,j_2,$ where $(\zeta,g^{\alpha_\ell}_{j_\ell}(\zeta))$ is a parametrization of the level set $\mathcal{C}_{\alpha_\ell}(\phi_\ell),$ and $J$ avoids the $z_1-$coordinates of the (possible) singularity. Then $C_{\Phi}:A^2_{\beta}(\mathbb{D}^2)\to A^2_{\beta}(\mathbb{D}^2)$ is not bounded, for all $\beta\ge-1.$
\end{theorem}
In the next Corollary we shall give an algebraic version of the previous Theorem. To fix notation, $\overline{q}$ is the reflection operator for the corresponding polynomial $q\in\mathbb{C}[z_1,z_2]$ defined in the bi-upper halfplane $\mathbb{H}^2:=\{z\in\mathbb{C}^2:\mathrm{Im}z_1>0,\mathrm{Im}z_2>0\}.$ 
\begin{corollary}\label{Puiseux2} Let $\Phi=(\phi_1,\phi_2)$ be two RIFs induced by stable polynomials $p_1,p_2\in\mathbb{C}[z_1,z_2]$ on the bidisc and denote by $q_1,q_2$ the corresponding polynomials defined in $\mathbb{H}^2,$ after passing from $\mathbb{D}^2$ to $\mathbb{H}^2$ with a Cayley transform. Assume that $\overline{q_1}-\alpha_1q_1$ and $\overline{q_2}-\alpha_2q_2$ share a common factor in their Weierstrass-Puiseux expansion for some pair $\vec{\alpha}=(\alpha_1,\alpha_2).$ Then the composition operator $C_{\Phi}:A^2_{\beta}(\mathbb{D}^2)\to A^2_{\beta}(\mathbb{D}^2)$ induced by RIFs $\phi_1=\widetilde{p}_1/p_1,\quad \phi_2=\widetilde{p}_2/p_2$ is not bounded, for all $\beta\ge-1$.
\end{corollary}

After the statement of the first two geometric non-boundedness criteria, a natural question arises: What exactly happens whenever two distinct curves or a straight line with a curve intersect? In the next two Theorems, we present the case in which the curves (or even a straight line with a curve) are tangential to each other. 
\begin{theorem}\label{smoothtangentlines}\textbf{(Smooth tangential intersection)} Let $\Phi=(\phi_1,\phi_2)$ where $\phi_1,\phi_2$ are RIFs and assume that there exist two analytic curves $(\zeta,g_{j_1}^{\phi_1,\alpha_1}(\zeta))\in\mathcal{C}_{{\alpha}_1}(\phi_1)$ and $(\zeta,g^{\alpha_2,\phi_2}_{j_2}(\zeta))\in\mathcal{C}_{\alpha_2}(\phi_2),$ for some $\alpha_1,\alpha_2\in\mathbb{T}$ such that {\color{refcyan}they intersect tangentially at a point $(\zeta_0,g_{j_1}^{\phi_1,\alpha_1}(\zeta_0))\in\mathbb{T}^2$, where $\zeta_0\in\mathbb{T}$, away from the (possible) singularities. Thus there exists an exponent $M>1$ such that} 
$$|g^{\alpha_1,\phi_1}_{j_1}(\zeta)-g_{j_2}^{\alpha_2,\phi_2}(\zeta)|\approx |\zeta-\zeta_0|^M$$
Then, the composition operator $C_{\Phi}:A^2_{\beta}(\mathbb{D}^2)\to A^2_{\beta}(\mathbb{D}^2)$ is not bounded. 
\end{theorem}

\begin{remark} The exponent $M>1$ is not to be confused with the \textit{order of contact} (see Section \ref{tools}) between two curves of the level set of a RIF. Here we have two different RIFs and curves that are tangential to each other after framing them on the same picture. We cannot speak about their order of contact as they do not belong to the level set of one RIF.
\end{remark}

As we observe in the assumptions, we consider the tangential intersection to occur in a point of smoothness of both RIFs. But what if this intersection happens on a common, singular point for two RIFs? The answer slightly differs and is contained in

\begin{theorem}\label{singlulartangentlines} \textbf{(Singular tangential intersection)} Let $\Phi=(\phi_1,\phi_2)$ {\color{refcyan}where $\phi_1,\phi_2$ are RIFs} with a common singularity $\omega=(\omega_1,\omega_2)\in\mathbb{T}^2$ and generic values $\alpha_1,\alpha_2$ on $\omega,$ satisfying Theorem 5.6. of \cite{Clark}. Assume that there exist two analytic curves $(\zeta,g_{j_1}^{\phi_1,\alpha_1}(\zeta))\in\mathcal{C}_{{\alpha}_1}(\phi_1)$ and $(\zeta,g^{\alpha_2,\phi_2}_{j_2}(\zeta))\in\mathcal{C}_{\alpha_2}(\phi_2)$ with respective contact orders $K_{\phi_1}, K_{\phi_2}\ge 2$ such that they intersect tangentially at $\omega\in\mathbb{T}^2$ with order $M>2K+1,$ where $K:=\max\{K_{\phi_1},K_{\phi_2}\}$ i.e.
$$|g^{\alpha_1,\phi_1}_{j_1}(\zeta)-g_{j_2}^{\alpha_2,\phi_2}(\zeta)|\approx |\zeta-\omega_1|^M$$
Then, the composition operator $C_{\Phi}:A^2_{\beta}(\mathbb{D}^2)\to A^2_{\beta}(\mathbb{D}^2)$ is not bounded for all $M>K(2\beta+4)+1$ or equivalently $-1<\beta<\frac{M-1}{2K}-2.$  
\end{theorem}
Our assumption on the $\alpha_1,\alpha_2\in\mathbb{T}$ satisfying Theorem 5.6. of \cite{Clark} will be explained in Section \ref{tools}.
\begin{remark} Let $\Phi=(\phi_1,\phi_2),$ where $\phi_1,\phi_2$ are RIFs with a common singularity $\omega\in\mathbb{T}^2.$ Boundedness of the composition operator $C_{\Phi}:A^2_{\beta}(\mathbb{D}^2)\to A^2_{\beta}(\mathbb{D}^2)$ cannot be characterized uniformly for all $\beta>-1.$ The singular tangential intersection Theorem implies that if such criterion exists, it will be quantitatively related to the contact orders $K_{\phi_1}, K_{\phi_2}$, the weight parameter $\beta$ and the order of tangential intersection $M.$ 
\end{remark}
The above Theorem and its proof can be modified for the half smooth-half singular case. If we assume that the first order conditions hold, and the level set at the singularity is tangential to a level curve of the smooth function locally (the \textit{implicit function Theorem} guarantees that such a curve exists ), then the order of tangential intersection $M$ should satisfy $M>(2\beta+4)K+1$ for the weighted Bergman space, where $K=\max\{K_{\phi},K_{\psi}\}$ where $K_{\psi}$ is the order of the tangential intersection of the local germ and the singularity $\omega\in\mathbb{T}^2.$

Thanks to these results, we also recover the main Theorem of \cite{Me} as a simple corollary, that is
\begin{corollary}
    Let $\Phi=(\phi,\phi)$ where $\phi$ is a RIF with a singularity in $\mathbb{T}^2$. Then $C_{\Phi}:A^2_{\beta}(\mathbb{D}^2)\to A^2_{\beta}(\mathbb{D}^2)$ is not bounded.
\end{corollary}
To see this, just observe that taking $\alpha_1=\alpha_2=\alpha$, forces us to have two copies of the same level set, which implies that their intersection contains all curves and all lines of the level set of $\phi.$ Then, by Theorems \ref{exceptionalline} or \ref{exceptionalcurve} the result follows.
 In general, the containment of the same straight line or arc of curve, or points of high order tangential intersection in two level sets of two different RIFs, seems to force volumes to concentrate on one direction. Intuitively speaking, the volumes of the sublevel sets $|\phi_j-1|$ when one direction is annihilated are forced to behave like a product of a one-dimensional shrinking disc of radius $\delta$ times a set of constant area, thus annihilating the influence of the second variable. This key observation forces the composition operator to be unbounded.
 
We present two figures, borrowed from \cite{Pisa}, to serve as examples of non boundedness emerging after observing two different level sets.
\begin{figure}[H]
\begin{center}
\includegraphics[width=12cm ]{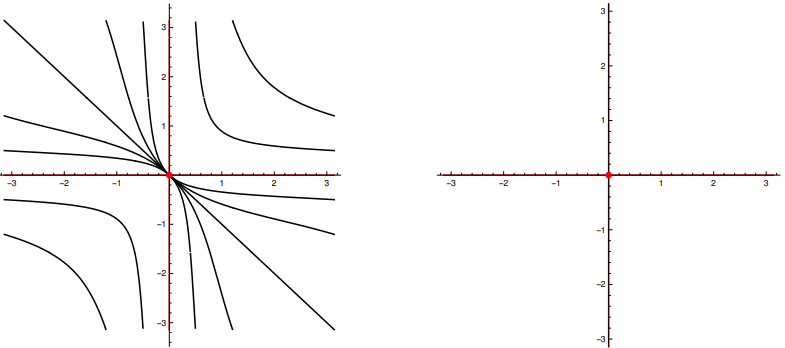}
\caption{A case in which $\mathcal{C}_{\alpha_1}(\phi_1)$ (left) and $\mathcal{C}_{\alpha_2}(\phi_2)$ (right) share common vertical and horizontal line. The composition operator in such a case is not bounded.}
\end{center}
\end{figure}
\begin{figure}[H]
\begin{center}
\includegraphics[width=12cm ]{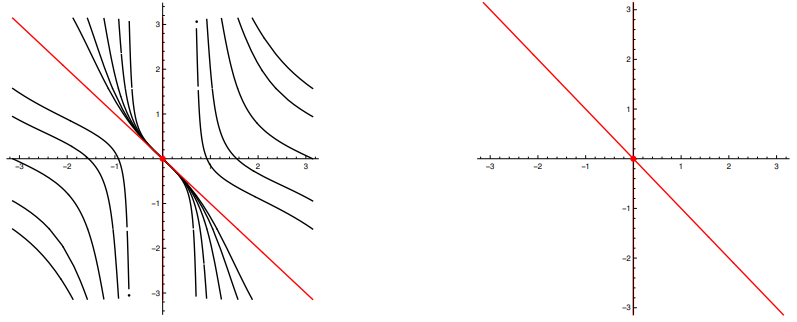}
\caption{A case in which $\mathcal{C}_{\alpha_1}(\phi_1)$ (left) and $\mathcal{C}_{\alpha_2}(\phi_2)$ (right) share a curve (anti-diagonal), $\gamma=(e^{it},e^{-it})\subset\mathbb{T}^2$ and a common vertical line. The composition operator in such a case is also not bounded.}
\end{center}
\end{figure}
Here is a brief summary of our results, to help the reader have a better picture of the geometric results.
\[
\boxed{
\begin{array}{|l|}
\hline
\begin{aligned}\\
&C_{\Phi}:A^2_{\beta}(\mathbb D^2)\to A^2_{\beta}(\mathbb D^2) \\[1mm]
&\textbf{\underline{Mixed case:} }\Phi=(\phi,\psi),\ 
\phi \text{ singular RIF},\ 
\psi\in \mathcal O(\mathbb D^2,\mathbb D)\cap \mathcal{C}^2(\overline{\mathbb D^2}) \\[0.5mm]
&\qquad \text{Tangential intersection at the singularity of order $M>4K+1$} \\[0.5mm]
&\qquad C_{\Phi}\ \text{bounded for }\beta\in(-1,0]
\iff
\mathcal{C}_{\lambda_1}(\varphi)\pitchfork \mathcal{C}_{\lambda_2}(\psi)\ \forall\,\lambda_1,\lambda_2\in\mathbb T.
\\[1.2mm]
\hline\\
&\textbf{\underline{Geometric obstructions:}} \\[0.5mm]
&\qquad \bullet \text{ common exceptional line } \Longrightarrow C_{\Phi}\ \text{not bounded},\\
&\qquad \bullet \text{common analytic arc } \Longrightarrow C_{\Phi}\ \text{not bounded},\\
&\qquad \bullet \text{ smooth tangential intersection } \Longrightarrow C_{\Phi}\ \text{not bounded},\\
&\qquad \bullet \text{ singular tangential intersection } \Longrightarrow C_{\Phi}\ \text{not bounded if } 
M>K(2\beta+4)+1,\\
&\qquad\qquad\text{equiv. } -1<\beta<\dfrac{M-1}{2K}-2,\qquad
K=\max\{K_{\phi_1},K_{\phi_2}\},\\\\
\end{aligned}
\\
\hline
\end{array}
}
\]
The notation "$\pitchfork$" means "transversal intersection". 
\section{Auxiliary results} \label{tools} 
 The new viewpoint that we provide in this paper is that the study of the geometry of the level set of a RIF can reveal operator theoretical properties of the composition operator. Some results on  compactness appear in recent papers \cite{compactness} and \cite{Bayart2} that give a first geometric link between operator theoretical properties and geometry of the symbol. Here, we choose to be a little bit more particular, and study the level sets of each coordinate function separately. Specifically, we focus on the level sets of RIFs. 
 
The theory on the RIF level sets has been developed in several papers. The reader is invited to study \cite{Pisa, AJM, Polonici, Clark1, Clark} for a comprehensive account on the level sets of RIFs and their connection with elements from algebraic geometry. In our paper we shall borrow results and definitions from these mentioned works, and we will apply them in a different context to match our purposes. To give a flavor, essentially, we consider two RIFs, picture their level sets, and then superimpose them to see how these two level sets interact. The characteristics of the interplay of the two level sets will reveal to us information on the composition operator induced by a self map {\color{refcyan}with coordinate functions equal to these two RIFs}, as it already has become evident from the statement of the main results.

Moreover, we provide some background material on the contact order and its rigorous definitions, along with the necessary local factorization results that will contribute in the proof of our main Theorem. In addition we give some background on the first order condition and its equivalent geometric statement.

\subsection{The level sets of RIFs}
Let us now provide the reader with the essential background through the definitions and theorems that will be invoked in the sequel. We begin with some definitions and results related to the level sets of RIFs, found in the paper \cite{Clark}.
\begin{definition} A point $\alpha\in\mathbb{T}$ is said to be an \textit{exceptional value} if $\varphi(\tau, z_2)=\alpha$ or if $\varphi(z_1,\tau) = \alpha$ for some $\tau\in\mathbb{T}.$ This is equivalent to saying that one of the two lines $\{\zeta \in \mathbb{T}^2: \zeta_1 = \tau\}$ or $\{\zeta \in \mathbb{T}^2: \zeta_2 = \tau\}$ is in $\mathcal{C}_{\alpha}(\varphi)$ for some $\tau\in\mathbb{T}.$ If $\alpha \in \mathbb{T}$ is not an
exceptional value, then we say that $\alpha$ is a generic value.
\end{definition}
Note that the graphs are implicit from the \textit{implicit function Theorem} away from the singularities. From the same paper, in order to reach the desired results, we shall invoke parts of Lemma 3.6. More precisely, a step of its proof. {\color{refcyan}If a RIF $\phi=\widetilde p/p$ has the line $\{\zeta\in\mathbb{T}^2:\zeta_1=\tau\}$ contained in $\mathcal C_{\alpha}(\phi)$, then, for this value of $\alpha\in\mathbb T$, the polynomial $\widetilde p-\alpha p$ vanishes on that line. It therefore vanishes identically as a polynomial in $z_2$ after setting $z_1=\tau$, and hence
$$\widetilde p(z_1,z_2)-\alpha p(z_1,z_2)=(z_1-\tau)Q(z_1,z_2),$$
where $Q$ is a polynomial in two variables.} For more details, the reader is invited to consult the proof of this specific Lemma in \cite{Clark}.

For the proof of the "common analytic curve Theorem" we will apply the following Lemma (Lemma 2.5. of \cite{Clark})

\begin{lemma} \label{parametrization}
Let $\phi$ be a bidegree $(m,n)\in\mathbb{N}^2$ RIF. For each $\alpha \in \mathbb{T}$ and any choice of $\tau_0\in \mathbb{T}$ there exist functions $g_{1}^{\alpha},\ldots ,g_{n}^{\alpha}$ defined on $\mathbb{T}$ and analytic on $\mathbb{T}\backslash \{\tau_0\}$ such that $\mathcal{C}_{\alpha}(\phi)$ can be written as a union of graphs of the form
$$\{(\zeta ,g_{j}^{\alpha}(\zeta)):\zeta \in \mathbb{T}\} ,\quad j = 1,\ldots ,n,$$
potentially, together with a finite number of vertical lines $\zeta_{1} = \tau_{1},\ldots ,\zeta_{1} = \tau_{k}$ , where each $\tau_{j}\in \mathbb{T}$.
\end{lemma}

For the proof of our characterization and for Corollary \ref{Puiseux2} after the common curve Theorem, it is necessary for us to recall the following Theorems related to the local
\textit{Weierstrass-Puiseux factorization} of stable polynomials. We shall invoke two results, namely

{\color{refcyan}After a rotation, suppose that the singularity is $(1,1)$ and that $p$ and $\widetilde p$ have no common factor. Write $\deg p=(d_1,d_2)$ and use the Cayley transform $m(z)=i(1-z)/(1+z)$, which maps $1$ to $0$. The polynomial in the next theorem is
$$q(w_1,w_2)=(1-iw_1)^{d_1}(1-iw_2)^{d_2} p\bigl(m^{-1}(w_1),m^{-1}(w_2)\bigr).$$
Thus $q$ has no zeros in $\mathbb H^2$ and $q(0,0)=0$. We write $\bar q(w)=\overline{q(\bar w_1,\bar w_2)}$. The factorization is local near the origin, and the variables are chosen as in \cite{Clark}.}

\begin{theorem}\label{Puiseux1} 
There exists a positive integer $J$ and associated positive integers
$M_1,\dots,M_J$ such that, near $(0,0)$, the polynomial $q$ admits the factorization
\begin{equation}\label{eq:theorem52-factorization}
q(z)=u(z)\prod_{j=1}^{J}\prod_{m=1}^{M_j}\left({\color{refcyan}z_2 + q_j(z_1) + z_1^{2L_j}\,\psi_j\!\bigl(\mu_j^{\,m} z_1^{1/M_j}\bigr)}
\right),
\end{equation}
where:
\begin{itemize}
    \item \(u\) is a unit at $(0,0)$, that is, $u$ is analytic and $u(0,0)\neq 0$;
    \item for each $j=1,\dots,J$, $q_j \in \mathbb{R}[z]$ is a polynomial with real coefficients satisfying $\deg q_j < 2L_j,$ $q_j(0)=0$ and $ q_j'(0)>0;$
    \item for each $j=1,\dots,J$, $\mu_j=e^{2\pi i/M_j}$ is a primitive $M_j$-th root of unity;
    \item for each $j=1,\dots,J$, $\psi_j$ is analytic in a neighborhood of $0$ and satisfies
    $\mathrm{Im} \psi_j(0) > 0.$
\end{itemize}
\end{theorem}
And
\begin{theorem}
{\color{refcyan}For all but finitely many $\alpha\in \mathbb{T}$, we can also factor}
$$\bar{q}-\alpha q=(1-\alpha)u^{\alpha}(z)\prod_{j=1}^J\prod_{m=1}^{M_j}\left({\color{refcyan}z_2+q_j(z_1)+z_1^{2L_j}\psi^{\alpha}_{j,m}(z_1)}\right)$$
where $u^{\alpha}$ is a unit at $(0,0)$, {\color{refcyan}each $q_j\in \mathbb{R}[z_1]$ with $q_j(0)=0$ and $q_j'(0)>0,$} and each $\psi^{\alpha}_{j,m}$ is a real analytic function in a neighborhood of the origin.
\end{theorem}

The idea here is that every analytic curve $g_{j}^{\alpha}$ of the level set {\color{refcyan}corresponds, after the change of variables, to the branch $w_2=-q_j(w_1)-w_1^{2L_j}\psi^{\alpha}_{j,m}(w_1)$.} This correspondence is possible via the Cayley transforms from $\mathbb{D}\to\mathbb{H}$ and its inverses, i.e.

$${\color{refcyan}m \colon \mathbb{D}\to \mathbb{H}, \quad m(z_j)=i\frac{1-z_j}{1+z_j},\qquad m^{-1}(w_j)=\frac{1+iw_j}{1-iw_j}, \quad j=1,2,}$$

It is important to note that in both factorizations, the $q_j$ polynomials match, but not the functions $\psi_j$ and $\psi^{\alpha}_{j,m}.$ They completely disagree, as $\mathrm{Im}\psi_j(0)>0$ and $\psi^{\alpha}_{j,m}(0)\in\mathbb{R}.$ This detail will be crucial in the sequel, specifically in the proof of our main result.

For more details, we suggest to the reader {\color{refcyan}to consult Subsection 5.2. in \cite{Clark}}.
\subsection{Contact order} In one of our main results in this paper, the notion of \textit{contact order} is mentioned. \textit{Contact order} is an algebro-geometric quantity related to the singularities of a RIF.

Consider a RIF $\phi$ of bidegree $(n,m)\in \mathbb{N}^2.$ The slice $$\phi(z_1,\zeta_2)=\phi_{\zeta_2}(z_1),\quad \zeta_2\in\mathbb{T}$$ is a finite Blaschke product with at most $n-$distinct zeros $a_1(\zeta_2),...,a_n(\zeta_2).$ The authors of the paper \cite{PLMS} proved, among other important results, the following.

\begin{theorem}(Theorem 3.5. of \cite{PLMS}) 
\label{def:loccont}
Let
$\varphi=\frac{\tilde p}{p}$ 
be a rational inner function on $\mathbb{D}^2$ with $\deg \varphi=(n,m)\in\mathbb{N}^2$ and a
singularity on $\mathbb{T}^2$ with $z_2$-coordinate $\tau_2$. Then there exists a
rational number $K>0$ such that
$$\epsilon(\varphi,\zeta_2):=\min_{1\leq j\leq n}\{1-|a_j(\zeta_2)|\}\approx |\tau_2-\zeta_2|^K,$$
for all $\zeta_2\in\mathbb{T}$ in a neighborhood $U$ of $\tau_2$. The number $K$ is
called the $(z_1,\tau_2)$-contact order of $\varphi$.
\end{theorem}
Due to this Theorem we obtain a definition of the notion of \textit{local contact order}.
The concrete definition of the local and global contact order follows

\begin{definition}(Definition 3.6 of \cite{PLMS}) Let $\varphi=\frac{\tilde p}{p}$ be a rational inner function on $\mathbb{D}^2$ with $\deg \varphi=(m,n)$ and let
$\tau_{2,1},\ldots,\tau_{2,J}$ denote the distinct $z_2$-coordinates of the singularities of
$\varphi$. Then the $z_1$-contact order of $\varphi$ is the number $K^1$ defined by

$$K^1:=\max\bigl\{(z_1,\tau_{2,j})\text{-contact order of }\varphi : 1\le j\le J\bigr\}.$$
\end{definition}

Thus the local contact order measures how rapidly a branch approaches the singularity in one variable, while the global contact order is the maximum of these local exponents over all singular points.

The above definition and Theorem \ref{def:loccont} are a consequence of the local Puiseux factorizations of the halfplane counterparts of our initial stable polynomials. Roughly speaking, the exponents $2L_j$ in our previous result on factorization, represent the \textit{local contact orders}, whereas the maximum of them is the \textit{contact order} of the RIF $K\ge2$. A specific property of the contact order is that it is always an even integer. The notion of the contact order will appear in the proof of Theorem \ref{singlulartangentlines}.

A useful fact for both tangential intersection Theorems (Theorems \ref{smoothtangentlines} and \ref{singlulartangentlines}) that the reader can find again in \cite{Clark} is that for generic values of a RIF, we have $$\frac{\partial{\phi_{\ell}}}{\partial z_2}(\zeta,g_j^{\alpha}(\zeta))=c\neq 0$$ away from the singularity, while $$\left|\frac{\partial\phi}{\partial z_2}(\zeta,g_j^{\alpha}(\zeta))\right|^{-1}\approx |\zeta-\omega_1|^{K_\phi}$$ when we approach the $z_1-$coordinate of the singularity, for all but finitely many $\alpha\in\mathbb{T}.$ The statement "for all but finitely many $\alpha\in\mathbb{T},$" is a subtle technicality that we need to include in the statement of our Theorem \ref{singlulartangentlines}.

\subsection{The first-order conditions: the works of Bayart and Kosi\'nski.}\label{smooth}
In the papers \cite{Bayart1} and \cite{CompositionsKosinski} the authors studied the problem of continuity of the composition operator whenever the self maps of the bidisc are smooth up to the closure of the bidisc. Both authors gave a characterization of smooth self maps that induce bounded composition operators. Let us restate the result in a geometric way.
\begin{proposition}\label{transverse}\textbf{(First-order condition is equivalent to transversality)}
Let $\Phi=(\phi,\psi)\in \mathcal O(\mathbb D^2,\mathbb D^2)\cap \mathcal{C}^2(\overline{\mathbb D^2}),$
and consider all $\zeta\in \mathbb T^2$ such that
$\Phi(\zeta)\in \mathbb T^2.$
{\color{refcyan}Assume that, for every such $\zeta$, the level sets
$\mathcal{C}_{\phi(\zeta)}(\phi)$ and $\mathcal{C}_{\psi(\zeta)}(\psi)$
are smooth curves locally at $\zeta$.}
Then the following are equivalent:
\begin{enumerate}
    \item The derivative $d_{\zeta}\Phi$ is invertible for all $\zeta\in\mathbb{T}^2$ such that $\Phi(\zeta)\in\mathbb{T}^2$.
    \item The boundary level curves
    $\mathcal{C}_{\phi(\zeta)}(\phi)\quad \mathrm{and}\quad
    \mathcal{C}_{\psi(\zeta)}(\psi)$
    are smooth for all $\zeta\in\mathbb{T}^2$ such that $\Phi(\zeta)\in\mathbb{T}^2$ and intersect transversely there.
    \item $C_{\Phi}:A^2_{\beta}(\mathbb{D}^2)\to A^2_{\beta}(\mathbb{D}^2)$ is bounded.
\end{enumerate}
\end{proposition}

\begin{proof}
{\color{refcyan}
We only need to prove (1) $\Longleftrightarrow$ (2), since (1) $\Longleftrightarrow$ (3) is the first-order condition. Fix $\zeta=(\zeta_1,\zeta_2)\in\mathbb T^2$ with $\Phi(\zeta)=(\alpha_1,\alpha_2)\in\mathbb T^2$. In angular coordinates near $\zeta$, write
\begin{align*}
\phi(\zeta_1e^{it_1},\zeta_2e^{it_2})&=\alpha_1r_1(t)e^{iu(t)},\\
\psi(\zeta_1e^{it_1},\zeta_2e^{it_2})&=\alpha_2r_2(t)e^{iv(t)},
\end{align*}
where $r_1,r_2,u,v$ are real-valued and $r_1(0)=r_2(0)=1$, $u(0)=v(0)=0$. Since $r_1,r_2\leq1$, both have a maximum at $0$. Hence $\nabla r_1(0)=\nabla r_2(0)=0$, and differentiating gives
$$D(u,v)(0)=\begin{pmatrix}1/\alpha_1&0\\0&1/\alpha_2\end{pmatrix}
d_\zeta\Phi\begin{pmatrix}\zeta_1&0\\0&\zeta_2\end{pmatrix}.$$
The Julia--Carath\'eodory theorem, applied to the diagonal slices
$$\lambda\longmapsto\frac{\phi(\lambda\zeta_1,\lambda\zeta_2)}{\alpha_1},
\qquad \lambda\longmapsto\frac{\psi(\lambda\zeta_1,\lambda\zeta_2)}{\alpha_2},$$
shows that $\nabla u(0)\ne0$ and $\nabla v(0)\ne0$. Indeed, these slices map $\mathbb D$ into itself, have value $1$ at $\lambda=1$, and their derivatives there are strictly positive and equal to the sums of the corresponding first-order partial derivatives of $u$ and $v$ at $0$.

Locally, $\mathcal C_{\alpha_1}(\phi)\subseteq\{u=0\}$ and $\mathcal C_{\alpha_2}(\psi)\subseteq\{v=0\}$. Since the level sets are smooth curves by assumption, their tangent lines at $\zeta$ are the kernels of $du(0)$ and $dv(0)$, respectively. They therefore intersect transversely if and only if $D(u,v)(0)$ is invertible. By the displayed matrix identity, this is equivalent to invertibility of $d_\zeta\Phi$.
}
\end{proof}

In his paper \cite{Bayart1}, Bayart also came up with a clever algorithm with which he was able to determine boundedness or not of the composition operator in higher dimensions, whenever the self map is induced by affine symbols. On the other hand, Kosi\'nski in his work provided a characterization for the unweighted Bergman space $A^2(\mathbb{D}^3)$ on the tridisc, again for smooth symbols. Later, in the paper \cite{Bayart2} Bayart characterized the smooth self maps of the tridisc on the Hardy space $H^2(\mathbb{D}^3),$ and provided several interesting new viewpoints in his more recent paper with Dorval \cite{Bayart3} on the weighted Bergman spaces of the tridisc $A^2_{\beta}(\mathbb{D}^3).$ 

A significant tool that appears in the previous work in the literature and is of great importance is the \textit{Julia-Carath\'eodory} Theorem, as formulated by Bayart in \cite{Bayart1}. We will also make a brief mention of this classical result {\color{refcyan}in one of the proofs of our main contributions}. Let us provide here its statement, as it appears in the paper of Bayart \cite{Bayart1}

\begin{lemma}
Let $\phi:\mathbb{D}\to\mathbb{D}$ be holomorphic $\mathcal{C}^{1}$ on $\mathbb{D} \cup \{\xi\}$, where $\xi \in \mathbb{T}$. Suppose moreover that $\phi(\xi)=\xi$. Then
$$\phi'(\xi)\in (0,+\infty).$$
\end{lemma}

For the purposes of our work, Theorem \ref{characterization} will be applied in the proof of Theorem \ref{characterization1}, specifically in the sufficiency part of it, whereas in the necessity part, we shall apply all of our geometric non-boundedness Theorems. For an application of the Julia-Carath\'eodory Theorem similar to the ideas of Kosi\'nski from the paper \cite{CompositionsKosinski}, see Section \ref{examples}.
\section{Proofs of the main results} \label{proofs} We divide the proofs into two parts. Section \ref{geometric} is devoted to the geometric non-boundedness results. Section \ref{transversal} shows the sufficiency part. The main result follows then as a corollary.
\subsection{Geometric non boundedness results}\label{geometric} We proceed with the geometric non-boundedness results.
Whenever convenient, after rotations, we may reduce to the case $\alpha_1=\alpha_2=1.$ {\color{refcyan}As above, the non-boundedness proofs use boxes with $\delta_1=\delta_2=\delta$; a single such family is sufficient for the conclusion.}
\begin{theorem}\textbf{(Common exceptional line)} Let $\phi_1,\phi_2$ be RIFs. If there exist $\alpha_1,\alpha_2\in\mathbb{T}$ such that both level sets $\mathcal{C}_{\alpha_1}(\phi_1)$ and $\mathcal{C}_{\alpha_2}(\phi_2)$ contain the same vertical line $\{\zeta\in \mathbb{T}^2:\zeta_1=\tau\}$ (and analogously, horizontal), then the composition operator $C_{\Phi}:A^2_{\beta}(\mathbb{D}^2)\to A^2_{\beta}(\mathbb{D}^2)$ with symbol $\Phi=(\phi_1,\phi_2)$ is not bounded.
\end{theorem}  
\begin{proof}
We will prove the Theorem in the case that both RIFs have a common line on the $z_1-$variable on their level sets $\mathcal{C}_{1}(\phi_\ell), \ell=1,2.$
We need to find one family of two-dimensional Carleson boxes $S(\zeta,\delta)$ such that 
$$\lim_{\delta\to0}\frac{V_{\beta}(\Phi^{-1}(S(\zeta,\delta)))}{V_{\beta}(S(\zeta,\delta))}=\infty.$$ 
Consider the box $S((1,1),C\delta)$ for $\delta\in(0,1).$ Our goal is to show that $V_{\beta}(\Phi^{-1}(S((1,1),C\delta)))>C'\delta^{2+\beta},$ deeming continuity of the composition operator impossible. 
It is clear that
$$\Phi^{-1}(S((1,1),C\delta))=\Bigg\{(z_1,z_2)\in\mathbb{D}^2:\left|\phi_1(z_1,z_2)-1\right|<C\delta,|\phi_2(z_1,z_2)-1|<C\delta\Bigg\}.$$

At this point, we claim that there exist polynomials $Q_1,Q_2$ in two variables, such that $$\phi_1(z_1,z_2)-1=\frac{(z_1-1)Q_1(z_1,z_2)}{p_1(z_1,z_2)}, \quad Q_1\in\mathbb{C}[z_1,z_2],$$
and
$$\phi_2(z_1,z_2)-1=\frac{(z_1-1)Q_2(z_1,z_2)}{p_2(z_1,z_2)}, \quad Q_2\in\mathbb{C}[z_1,z_2],$$

\textbf{Proof of the claim:}
From Lemma 3.6. in the paper \cite{Clark}, our assumption essentially means that both RIFs have 1 as an exceptional value on the same variable. This implies that the sets $$\mathcal{C}_1(\phi_\ell)=\mathrm{clos}\{\zeta\in\mathbb{T}^2:\phi^*_{\ell}(\zeta)=1\},\quad \ell=1,2$$ both contain the line $\{1\}\times \mathbb{T},$ and moreover from page 14 of the same paper, we see that there exist $Q_1,Q_2\in \mathbb{C}[z_1,z_2]$ such that

$$\phi_1(z_1,z_2)-1=\frac{(z_1-1)Q_1(z_1,z_2)}{p_1(z_1,z_2)} \quad \mathrm{and} \quad \phi_2(z_1,z_2)-1=\frac{(z_1-1)Q_2(z_1,z_2)}{p_2(z_1,z_2)}, \quad (z_1,z_2)\in\mathbb{D}^2. $$

{\color{refcyan}Indeed, for $\ell=1,2$ the polynomial $\widetilde p_\ell(1,z_2)-p_\ell(1,z_2)$ vanishes for infinitely many $z_2\in\mathbb T$ and is therefore identically zero. Consequently $z_1-1$ divides $\widetilde p_\ell-p_\ell$. Away from the singularities the quotients are smooth; in particular, on a sufficiently small neighborhood of $\{1\}\times\{|z_2|\le 1/2\}$ the denominators are bounded away from zero.}

Pick now $\delta$ sufficiently small, and consider the set
$$A_{\delta}:=\{z\in\mathbb{D}^2:|z_1-1|<\delta,|z_2|<1/2\}.$$

We observe that the polynomials $p_1,p_2$ do not vanish on $A_\delta,$ due to stability (both polynomials vanish only on the two-torus for RIFs, not inside the bidisc). As a consequence, we receive constants $C_1,C_2>0$ such that
$$|\phi_1(z_1,z_2)-1|=\left|\frac{(z_1-1)Q_1(z_1,z_2)}{p_1(z_1,z_2)}\right|\leq C_1\delta, \quad z\in A_{\delta},$$
and moreover
$$|\phi_2(z_1,z_2)-1|=\left|\frac{(z_1-1)Q_2(z_1,z_2)}{p_2(z_1,z_2)}\right|\leq C_2\delta, \quad z\in A_{\delta}.$$
By setting $C=\max\{C_1,C_2\}$ we observe that
$A_{\delta}\subset\Phi^{-1}(S((1,1),C\delta)).$
Now 
\begin{align}
V_{\beta}(\Phi^{-1}(S((1,1),C\delta)))
&>V_{\beta}(A_{\delta})&&\\
&>\int_{\{z\in\mathbb{D}^2:|z_1-1|<\delta, |z_2|<1/2\}}(1-|z_1|^2)^{\beta}(1-|z_2|^2)^{\beta}dV(z_1,z_2) \notag&&\\
&\gtrsim\delta^{2+\beta}, \quad \forall \delta\in (0,1). \notag
\end{align}
where the implied constant is depending only on $\beta\ge -1.$
This finishes the proof.
\end{proof}
Next, we re-state and sketch the proof of the mixed case, for convenience.
\begin{theorem} Let $\Phi=(\phi,\psi),$ where $\phi$ is a RIF with exceptional value $\alpha$ on the $z_1-$variable, $\psi\in \mathcal{O}(\mathbb{D}^2,\mathbb{D})\cap \mathcal{C}^2(\overline{\mathbb{D}^2})$ such that $\partial_{z_2}\psi=0$ and $\psi(1,1)=1.$ Then the composition operator $C_{\Phi}:A^2_{\beta}(\mathbb{D}^2)\to A^2_{\beta}(\mathbb{D}^2)$ is not bounded for all $\beta\ge-1.$
\end{theorem}
\begin{proof} Again, after a normalization $\alpha=1,$ using the same $A_{\delta}-$set as before, we have
$$|\phi_1(z_1,z_2)-1|=\left|\frac{(z_1-1)Q_1(z_1,z_2)}{p_1(z_1,z_2)}\right|\leq C_1\delta, \quad z\in A_{\delta}.$$
By our assumption on the partial of $\psi$ and the Julia-Carath\'eodory Theorem, we have that $\psi(1,z_2)=1.$ From smoothness of $\psi$ we get
$$|\psi(z_1,z_2)-1|=|\psi(z_1,z_2)-\psi(1,z_2)|\leq C_2|z_1-1|,\quad z\in A_{\delta}.$$
Then we proceed as before.
\end{proof}
\begin{theorem}\textbf{(Common analytic curve)} Let $\Phi=(\phi_1,\phi_2)$ induced by RIFs $\phi_\ell=\frac{\widetilde{p}_\ell}{p_\ell}$, where $\ell=1,2.$
Assume that there exists  $\alpha_1,\alpha_2\in\mathbb{T}$ and a closed arc $J\subset \mathbb{T}$ of positive length such that $g^{\alpha_1}_{j_1}(\zeta)=g^{\alpha_2}_{j_2}(\zeta)=g(\zeta)$, for $\zeta\in J$ and some indices $j_1,j_2,$ where $(\zeta,g(\zeta))$ is a parametrization of the level set $\mathcal{C}_{\alpha_\ell}(\phi_\ell),$ and $J$ avoids the $z_1-$coordinates of the (possible) singularity. Then $C_{\Phi}:A^2_{\beta}(\mathbb{D}^2)\to A^2_{\beta}(\mathbb{D}^2)$ is not bounded, for all $\beta>-1.$
\end{theorem}
\begin{remark} There are concrete examples that show that $\mathcal{C}_{\alpha_1}(\phi_1)\cap\mathcal{C}_{\alpha_2}(\phi_2)=\{(\zeta,g(\zeta))\}$ can occur for two different RIFs. See pages 454-455 in the paper \cite{Pisa}.
\end{remark}
\begin{proof}
The arc $J$ avoids the $z_1-$coordinates of the possible singularities of $\phi_\ell, \ell=1,2,$ hence, along the assumed common graph $\Gamma:=\{(\zeta,g(\zeta)),\zeta\in J\}.$ We know that  $$\frac{\partial{\phi_{\ell}}}{\partial z_2}(\zeta,g(\zeta))\neq 0, \quad \zeta\in J $$
Moreover, on $\mathcal{C}_{\alpha_{\ell}}(\phi_{\ell})$ for $\ell=1,2,\quad \phi_{\ell}(\zeta,g(\zeta))=\alpha_{\ell}.$ Since RIFs do not possess any singularities in $\mathbb{T}\times \mathbb{D} \quad \mathrm{or}\quad\mathbb{D}\times\mathbb{T}$ (this in fact comes again by the stability of the polynomials $p_{\ell}$), there exists a $\rho_0>0$ such that in the set
$$K_{\rho_0}:=\Bigg\{(r\zeta,z_2)\in\mathbb{D}^2:\zeta\in J,1-\rho_0\leq r\leq 1,|z_2-g(\zeta)|<\rho_0\Bigg\}$$ the partial derivatives of both RIFs are bounded from above, i.e.
$$M_{1,\ell}=\sup_{K_{\rho_0}}\left|\frac{\partial \phi_{\ell}}{\partial z_1}\right|<\infty \quad \mathrm{and} \quad M_{2,\ell}=\sup_{K_{\rho_0}}\left|\frac{\partial \phi_{\ell}}{\partial z_2}\right|<\infty.$$
This implies that
$$|\phi_{\ell}(\zeta,z_2)-\alpha_{\ell}|=\left|\int_{z_2}^{g(\zeta)}\frac{\partial{\phi}_{\ell}}{\partial z_2}(\zeta,w)dw\right|\lesssim|z_2-g(\zeta)|,\quad (1).$$
Moreover,
$$|\phi_{\ell}(r\zeta,z_2)-\phi_{\ell}(\zeta,z_2)|=\left|\int_{r}^{1}\zeta \frac{\partial \phi_{\ell}}{\partial z_1}(t\zeta,z_2)dt\right|\lesssim M_{\ell,1}(1-r), \quad (2)$$
By (1) and (2) and the triangle inequality, we obtain
$$|\phi_\ell(r\zeta,z_2)-\alpha_\ell|\lesssim M_{1,\ell}(1-r)+|z_2-g(\zeta)|\lesssim\delta, \quad \forall z\in A_{\delta},$$
where $$A_{\delta}:=\Bigg\{(r\zeta,z_2)\in\mathbb{D}^2:\zeta\in J:1-r<\delta/C,|z_2-g(\zeta)|<\delta/C\Bigg\}$$
where $C=\max\{M_{i,j}\}>0$, and $i,j=1,2.$ This, in turn, implies that
$$A_{\delta}\subset \Phi^{-1}(S(\alpha_1,\alpha_2),C\delta).$$ By Fubini's Theorem, we observe
\begin{align}
V_{\beta}(\Phi^{-1}(S((\alpha_1,\alpha_2),C\delta)))\gtrsim \notag& V_{\beta}(A_{\delta})&&\\ 
=&\left(\int_{\{r\zeta:\zeta\in J,1-r<\delta/C\}}dA_{\beta}(z_1)\right)\left(\int_{\{z\in\mathbb{D}:|z_2-g(\zeta)|<\delta/C\}}dA_{\beta}(z_2)\right) \notag&&\\ \notag
\approx&\delta^{\beta+1}\delta^{\beta+2}&&\\
\approx&\delta^{2\beta+3}. 
\end{align}
This suffices to prove that the composition operator is not bounded.
\end{proof}
\begin{remark} For the case of mixed symbols $\Phi=(\phi,\psi)$ where $\phi$ is a RIF and $\psi\in\mathcal{O}(\mathbb{D}^2,\mathbb{D})\cap \mathcal{C}^{2}(\overline{\mathbb{D}^2})$, we consider $g$ to be the analytic curve that parametrizes the level set of $\psi$ via the \textit{implicit function Theorem} and we proceed similarly with the estimates.
\end{remark}
A corollary we receive is
\begin{corollary} Let $\Phi=(\phi_1,\phi_2)$ be two RIFs induced by stable polynomials $p_1,p_2\in\mathbb{C}[z_1,z_2]$ on the bidisc and denote by $q_1,q_2$ the corresponding polynomials defined in $\mathbb{H}^2,$ after passing from $\mathbb{D}^2$ to $\mathbb{H}^2$ with a Cayley transform. Assume that $\overline{q_1}-\alpha_1q_1$ and $\overline{q_2}-\alpha_2q_2$ share a common factor in their Weierstrass-Puiseux expansion for some pair $\vec{\alpha}=(\alpha_1,\alpha_2).$ Then the composition operator $C_{\Phi}:A^2_{\beta}(\mathbb{D}^2)\to A^2_{\beta}(\mathbb{D}^2)$ induced by RIFs $\phi_1=\widetilde{p}_1/p_1,\quad \phi_2=\widetilde{p}_2/p_2$ is not bounded, for all $\beta\ge-1$.
\end{corollary}
We now turn to the smooth tangential intersection non boundedness result.

\begin{theorem}\textbf{(Smooth tangential intersection)} Let $\Phi=(\phi_1,\phi_2)$ {\color{refcyan}where $\phi_1,\phi_2$ are RIFs} and assume that there exist two analytic curves $(\zeta,g_{j_1}^{\phi_1,\alpha_1}(\zeta))\in\mathcal{C}_{{\alpha}_1}(\phi_1)$ and $(\zeta,g^{\alpha_2,\phi_2}_{j_2}(\zeta))\in\mathcal{C}_{\alpha_2}(\phi_2),$ for some $\alpha_1,\alpha_2\in\mathbb{T}$ such that {\color{refcyan}they intersect tangentially at a point $(\zeta_0,g_{j_1}^{\phi_1,\alpha_1}(\zeta_0))\in\mathbb{T}^2$, where $\zeta_0\in\mathbb{T}$, away from the (possible) singularities. Thus there exists an exponent $M>1$ such that} 
$$|g^{\alpha_1,\phi_1}_{j_1}(\zeta)-g_{j_2}^{\alpha_2,\phi_2}(\zeta)|\approx |\zeta-\zeta_0|^M$$
Then, the composition operator $C_{\Phi}:A^2_{\beta}(\mathbb{D}^2)\to A^2_{\beta}(\mathbb{D}^2)$ is not bounded. 
\end{theorem}
\begin{proof} The strategy for the proof of our previous result will be followed after some necessary changes. {\color{refcyan}Here $I\subset\mathbb T$ is a sufficiently small arc containing $\zeta_0$ on which both of the displayed graphs are defined and avoid the singularities.} We choose to include this proof separately to emphasize here that the two curves do not share an arc; they meet only on a point of tangential intersection. Moreover, the tangential intersection is away from the singularity, therefore, we can apply local arguments on the partials and take $\partial_{z_i}\phi_{\ell}(\zeta,g_{j_\ell}^{\alpha_{\ell},\phi_{\ell}}(\zeta))=c_{i\ell}\neq0,$ for $\ell=1,2.$ Therefore, similar estimates to our previous result hold. More specifically

By applying the same ideas, we obtain two local estimates

$$|\phi_1(r\zeta,z_2)-\alpha_1|\lesssim (1-r)+|z_2-g^{\alpha_1,\phi_1}_{j_1}(\zeta)|,$$
and
\begin{align}
|\phi_2(r\zeta,z_2)-\alpha_2|&\lesssim (1-r)+|z_2-g^{\phi_2,\alpha_2}_{j_2}(\zeta)| \notag&&\\
&\lesssim(1-r)+|g_{j_1}^{\phi_1,\alpha_1}(\zeta)-g_{j_2}^{\phi_2,\alpha_2}(\zeta)|+|z_2-g_{j_1}^{\phi_1,\alpha_1}(\zeta)|\notag&&\\
&\approx (1-r)+|\zeta-\zeta_0|^M+|z_2-g^{\phi_1,\alpha_1}_{j_1}(\zeta)|
.\end{align}
In a similar manner to the previous proof we need to consider the set

$$A_{\delta}:=\Bigg\{(r\zeta,z_2)\in\mathbb{D}^2:\zeta\in I,\quad|\zeta-\zeta_0|^M<C_1\delta,\quad 1-r<C_2\delta, \quad|z_2-g_{j_1}^{\phi_1,\alpha_1}(\zeta)|<\delta\Bigg\},$$

and observe $A_{\delta}\subset \Phi^{-1}(S((\alpha_1,\alpha_2),\delta)).$
We will now estimate the volume of $A_{\delta}.$

\begin{align}
V_{\beta}(\Phi^{-1}(S((\alpha_1,\alpha_2),\delta)))>V_{\delta}(A_{\beta})=&\int_{\{(r\zeta,z_1)\in\mathbb{D}^2:\zeta\in I,|\zeta-\zeta_0|^M<C_1\delta,1-r<C_2\delta, |z_2-g^{\alpha_1,\phi_1}_{j_1}(\zeta)|<\delta\}}dV_{\beta}(z_1,z_2)\notag&&\\
=&\left(\int_{\{|\zeta-\zeta_0|^M<C_1\delta,1-r<C_2\delta\}}dA_{\beta}(z_1)\right)\left(\int_{|z_2-g^{\alpha_1,\phi_1}_{j_1}(\zeta)|<\delta}dA_{\beta}(z_2)\right)\notag&&\\
\approx &\delta^{1/M}\delta^{\beta+1}\delta^{\beta+2}&&\\
\approx &\delta^{2\beta+3+\frac{1}{M}} \notag
\end{align}
Since $M>1,$ this suffices to show non-boundedness of the composition operator.
\end{proof}
\begin{remark} As before, for the mixed case of symbols, the only change is that we should replace one of the $g^{\alpha_\ell,\phi_\ell}_{j_{\ell}}(\zeta)$ with $g,$ where g is a smooth local parametrization of the level set of $\psi\in\mathcal{O}(\mathbb{D}^2,\mathbb{D})\cap \mathcal{C}^2(\overline{\mathbb{D}^2}),$ coming again from the \textit{implicit function Theorem}.
\end{remark}
\begin{remark} As we stated in Section \ref{results}, the number $M$ is not to be confused with the contact order (or order of contact) of these two analytic curves. In principle, the level curves we assume that intersect tangentially are coming from two separate RIFs, and moreover the intersection happens in a point of the bitorus where both functions are smooth.
\end{remark}
Now we pass to the proof of Theorem \ref{singlulartangentlines}. Before we proceed with the details, let us comment on its significance. Out of all the geometric non-boundedness conditions, this Theorem is the most unique due to the fact that non-boundedness is quantitatively related to the contact orders $K_{{\phi_1}},K_{\phi_{2}}\ge2$ and the order in which these two branches are tangential to each other.
This fact is what prevents us from obtaining a full boundedness criterion for all geometric cases, working uniformly for all $\beta>-1.$ 

Note that from our statement, one could be tempted to assume that $M=\max\{K_{\phi_1},K_{\phi_2}\}.$ This would be true if we had the contact orders of the branches coming from one RIF. In our setting the branches come from two different RIFs, and nothing excludes the possibility of coefficients of the difference $g^{\alpha_1,\phi_1}_{j_1}(\zeta)-g_{j_2}^{\alpha_2,\phi_2}(\zeta)$ canceling each other in higher order than $K.$ Note that our criterion cannot decide non-boundedness for smaller $M$ than the stated.
\begin{theorem}\textbf{(Singular tangential intersection)} Let $\Phi=(\phi_1,\phi_2)$ {\color{refcyan}where $\phi_1,\phi_2$ are RIFs} with a common singularity $\omega=(\omega_1,\omega_2)\in\mathbb{T}^2$ and generic values $\alpha_1,\alpha_2$ on $\omega,$ satisfying Theorem 5.6. of \cite{Clark}. Assume that there exist two analytic curves $(\zeta,g_{j_1}^{\phi_1,\alpha_1}(\zeta))\in\mathcal{C}_{{\alpha}_1}(\phi_1)$ and $(\zeta,g^{\alpha_2,\phi_2}_{j_2}(\zeta))\in\mathcal{C}_{\alpha_2}(\phi_2)$ with respective contact orders $K_{\phi_1}, K_{\phi_2}\ge 2$ such that they intersect tangentially at $\omega\in\mathbb{T}^2$ with order $M>2K+1,$ where $K:=\max\{K_{\phi_1},K_{\phi_2}\}$ i.e.
$$|g^{\alpha_1,\phi_1}_{j_1}(\zeta)-g_{j_2}^{\alpha_2,\phi_2}(\zeta)|\approx |\zeta-\omega_1|^M$$
Then, the composition operator $C_{\Phi}:A^2_{\beta}(\mathbb{D}^2)\to A^2_{\beta}(\mathbb{D}^2)$ is not bounded for all $M>K(2\beta+4)+1$ or equivalently $-1<\beta<\frac{M-1}{2K}-2.$  
\end{theorem}
\begin{proof} First, we pick $\alpha_1,\alpha_2$ such that they satisfy Theorem 5.6 of \cite{Clark}, otherwise we cannot proceed with the proof. {\color{refcyan}For each RIF, the excluded set of unimodular values is finite. For a value outside this set and a level-set branch through the singularity, the cited theorem states that the corresponding Clark weight is comparable to $|\zeta-\omega_1|^{K_{\phi_\ell}}$. On the branch, this weight is the reciprocal of $|\partial_{z_2}\phi_\ell|$.} Such $\alpha_1,\alpha_2$ always exist, therefore our assumption is to follow the details of the mentioned paper but have no significant role here. The delicate part of the proof that distinguishes this case from the previous tangential intersection Theorem, is that we cannot use here the fact that
$$\partial_{z_i}\phi_{\ell}(\zeta,g^{a_{\ell},\phi_\ell}_{j_\ell}(\zeta))=c_{i\ell}\neq 0, \quad i=1,2.$$
This is the key assumption that previously allowed us to obtain the local upper bounds on $|\phi_{\ell}-\alpha_{\ell}|,\ell=1,2.$ Having clarified this, after using the assumption of generic values on the $z_1-$coordinate of the singularity, we know from the discussion in Section 5.1. and Theorem 5.6 of the paper \cite{Clark}, that the corresponding Clark weights satisfy $$W_{j_{\ell}}^{\alpha_\ell}(\zeta)\approx\left|\frac{\partial\phi_{\ell}}{\partial z_2}(\zeta,g_{j_\ell}^{\alpha}(\zeta))\right|^{-1}\approx |\zeta-\omega_1|^{K_{\phi_\ell}},\quad \ell=1,2,$$ close to the $z_1-$coordinate of the singularity.  This essentially means that
$$|\partial_{z_i}\phi_{\ell}(\zeta,g^{a_{\ell},\phi_\ell}_{j_\ell})(\zeta)|\approx |\zeta-\omega_1|^{-K_{\phi_\ell}}\quad \ell=1,2,\quad i=1,2.$$ With similar ideas as the previous result, we now obtain the estimates

$$|\phi_1(r\zeta,z_2)-\alpha_1|\lesssim \frac{1}{|\zeta-\omega_1|^{K_{\phi_1}}}\left((1-r)+|z_2-g^{\alpha_1,\phi_1}_{j_1}(\zeta)|\right),$$
and
\begin{align}
|\phi_2(r\zeta,z_2)-\alpha_2|&\lesssim \frac{1}{|\zeta-\omega_1|^{K_{\phi_2}}}\left((1-r)+|z_2-{\color{refcyan}g^{\phi_2,\alpha_2}_{j_2}(\zeta)}|\right) \notag&&\\
&\lesssim\frac{1}{|\zeta-\omega_1|^{K_{\phi_2}}}\left((1-r)+|g_{j_1}^{\phi_1,\alpha_1}(\zeta)-g_{j_2}^{\phi_2,\alpha_2}(\zeta)|+|z_2-g_{j_1}^{\phi_1,\alpha_1}(\zeta)|\right)\notag&&\\
&\approx \frac{1}{|\zeta-\omega_1|^{K_{\phi_2}}}\left((1-r)+|\zeta-\omega_1|^M+{\color{refcyan}|z_2-g^{\phi_1,\alpha_1}_{j_1}(\zeta)|}\right)
.\end{align}
The idea is now to find the proper $A_{\delta}-$set that will be subset of $\Phi^{-1}(S((\alpha_1,\alpha_2),C\delta)).$ Set $s:=|\zeta-\omega_1|$ for a moment.
The natural candidate for this is

$$A_{\delta}:=\Bigg\{z=(r\zeta,z_2)\in\mathbb{D}^2,\zeta\in I: s^{M-K}<C\delta,1-r<C\delta s^{K},|z_2-g_{j_1}^{\alpha_1,\phi_1}(\zeta)|<C\delta s^K \Bigg\}.$$

Of course we observe that with this definition, $A_{\delta}\subset \Phi^{-1}(S((\alpha_1,\alpha_2),C\delta))$ for some $C>0.$ So one just needs to calculate the volume $V_{\beta}(A_{\delta}).$ By Fubini's Theorem

\begin{align} V_{\beta}(A_{\delta})\approx&\int_{\{
|\zeta-\omega_1|^{M-K}<\delta\}} \left(\int_{1-\delta|\zeta-\omega_1|^K}^{1}(1-r^2)rdr\right)\left(\int_{|\zeta_2-g_{j_1}^{\phi_1,\alpha_1}(\zeta)|<\delta |\zeta-\omega_1|^{K}}dA_{\beta}(z_1)\right)|d\zeta|\notag&&\\
\approx& \int_{\{
|\zeta-\omega_1|^{M-K}<\delta\}}(\delta|\zeta-\omega_1|^{K})^{\beta+1}\cdot(\delta|\zeta-\omega_1|^K)^{\beta+2}|d\zeta|\notag&&\\
\approx&\delta^{2\beta+3}\int_{\{
|\zeta-\omega_1|^{M-K}<\delta\}}|\zeta-\omega_1|^{K(2\beta+3)}|d\zeta|\notag&&\\
\approx&\delta^{2\beta+3}\int_{0}^{\delta^{1/(M-K)}}t^{K(2\beta+3)}dt\notag&&\\
\approx&\delta^{2\beta+3+\frac{K(2\beta+3)+1}{M-K}}.
\end{align}
Comparing with $\delta^{2\beta+4},$ this tends to infinity exactly when $-1+\frac{K(2\beta+3)+1}{M-K}<0,$ or equivalently {\color{refcyan}when $K(2\beta+3)+1<M-K$ or equivalently} $M>K(2\beta+4)+1$ and solving for $\beta$ this gives $\beta<\frac{M-1}{2K}-2,$ which makes sense only whenever $M>2K+1.$ This finishes the proof.
\end{proof}
\begin{remark} For the mixed symbols, we replace one of the $g^{\alpha_\ell,\phi_\ell}_{j_\ell}(\zeta)$ with the analytic germ that parametrizes locally the level set of $\psi\in \mathcal{O}(\mathbb{D}^2,\mathbb{D})\cap\mathcal{C}^2(\mathbb{D}^2),$ and we moreover assume that it approaches the singularity with a corresponding $K_{\psi}-$order. The proof proceeds similarly.
\end{remark}
\subsection{Transversality implies boundedness}\label{transversal} 
Here we prove the boundedness sufficiency Theorem and obtain our main result as a Corollary.
\begin{theorem} Let $\Phi=(\phi,\psi)$ with $\phi=\widetilde{p}/p$ with a singularity at $(1,1),$ a RIF and $\alpha_1\in\mathbb{T}$ is the generic value at its singularity. Let $\psi\in\mathcal{O}(\mathbb{D}^2,\mathbb{D})\cap \mathcal{C}^2(\overline{\mathbb{D}^2})$ with $\psi(1,1)=\tau\in\mathbb{T}.$ Assume that the level sets $\mathcal{C}_{\alpha_1}(\phi)$ is locally transversal to the level set $\mathcal{C}_{\tau}(\psi),$ and for all other points $\zeta\in\mathbb{T}^2$ s.t. $\Phi(\zeta)\in\mathbb{T}^2,$ first order conditions hold.  Then, the composition operator $C_{\Phi}:A^2_{\beta}(\mathbb{D}^2)\to A^2_{\beta}(\mathbb{D}^2)$ is bounded, for all $\beta\in(-1,0].$
\end{theorem}
\begin{proof}
Transversal intersection in smooth points is equivalent to the first order conditions. Hence, the volumes are already bounded from above by $\approx\delta_1^{2+\beta}\delta_2^{2+\beta}.$
We only need to check the singularity and how the volume of inverse image of a Carleson box under $\Phi,$ behaves there.

Let $K_\nu\ge 2$ be the contact orders of each factor. We will initially handle the problem on the bi-upper halfplane $\mathbb{H}^2.$ Set $q,\bar{q}$ for the corresponding polynomials after passing from the bidisc to the halfplane via $m:\mathbb{D}\to\mathbb{H}$ applied on both variables. {\color{refcyan}In the following factorization, $\alpha$ denotes the value $\alpha_1$ from the theorem, and $U$ below will denote a sufficiently small neighborhood of the singularity $(1,1)$ in the bidisc coordinates.} The assumption on generic value and the local factorization Theorems imply that
$$\overline{q}(w_1,w_2)-\alpha q(w_1,w_2)=(1-\alpha)u_{\alpha}(w_1,w_2)\prod_{\nu=1}^N\prod_{m}^{M_\nu}(w_2+q_\nu(w_1)+w_1^{K_\nu}\psi^{\nu,m}_{\alpha}(w_1))\quad \mathrm{and}$$
$$q(w_1,w_2)=u(w_1,w_2)\prod_{\nu=1}^N\prod_{m}^{M_\nu}(w_2+q_\nu(w_1)+w_1^{K_\nu}\psi_\nu(\mu_\nu^mw_1^{1/M_{\nu}})),$$
{\color{refcyan}with $\psi^{\nu,m}_{\alpha}$ real analytic, $\psi^\nu$ analytic, and $\psi_{\alpha}^{\nu,m}(0)\in\mathbb{R}$.} Then, if $\Psi=\bar{q}/q,$
\begin{align}  |\Psi(w_1,w_2)-\alpha|\approx& \frac{|\prod_{\nu=1}^N\prod_{m}^{M_\nu}(w_2+q_\nu(w_1)+w_1^{K_\nu}\psi^{\nu,m}_{\alpha}(w_1))|}{|\prod_{\nu=1}^N\prod_{m}^{M_\nu}(w_2+q_\nu(w_1)+w_1^{K_\nu}\psi_\nu(\mu_\nu^mw_1^{1/M_{\nu}}))|}\notag&&\\
\approx& \frac{|\prod_{\nu=1}^N\prod_{m}^{M_\nu}(w_2+q_\nu(w_1)+w_1^{K_\nu}\psi^{\nu,m}_{\alpha}(w_1))|}{|\prod_{\nu=1}^N\prod_{m}^{M_\nu}(w_2+q_\nu(w_1)+w_1^{K_\nu}\psi_\nu(\mu_\nu^mw_1^{1/M_{\nu}})-w_1^{K_\nu}\psi^{\nu,m}_{\alpha}(w_1)+w_1^{K_\nu}\psi^{\nu,m}_{\alpha}(w_1))|}\notag&&\\
\gtrsim& \prod_{\nu=1}^N\prod_{m}^{M_\nu}\frac{|w_2+q_\nu(w_1)+w_1^{K_{\nu}}\psi^{\nu,m}_{\alpha}(w_1)|}{|w_2+q_\nu(w_1)+w_1^{K_\nu}\psi_{\alpha}^{\nu,m}(w_1)|+|w_1|^{K_\nu}|\psi^{\nu,m}_{\alpha}(w_1)-\psi_\nu(\mu_{\nu}^mw_1^{1/M_\nu})|}
\end{align}
by triangle inequality. Moreover, $\psi_{\alpha}^{\nu,m}(0)-\psi_{\nu}(0)\neq0$ because each $\psi_{\alpha}^{\nu,m}(0)\in\mathbb{R}$ and $\mathrm{Im}(\psi_\nu)(0)>0,$ hence, as $w_1\to 0,$ $|\psi^{\nu,m}_{\alpha}(w_1)-\psi_\nu(\mu_{\nu}^mw_1^{1/M_j})|$ does not vanish in a neighborhood of $0.$ This implies that there exists a constant $C_{\nu,m}>0$ such that
\begin{align}|\Psi(w_1,w_2)-\alpha|\gtrsim\prod_{\nu=1}^N\prod_{m}^{M_\nu}\frac{|w_2+q_\nu(w_1)+w_1^{K_{\nu}}\psi^{\nu,m}_{\alpha}(w_1)|}{|w_2+q_\nu(w_1)+w_1^{K_\nu}\psi_{\alpha}^{\nu,m}(w_1)|+|w_1|^{K_\nu}C_{\nu,m}}.\label{disgustingestimate}\end{align}
Set $h^{\alpha,\nu,m}_{\mathbb{H}^2}(w_1)=-q_\nu(w_1)-w_1^{K_{\nu}}\psi_{\alpha}^{\nu,m}(w_1),$ and on the bidisc $h^{\alpha_1,\nu,m}_{\mathbb{D}^2}(z_1),$ denoting the corresponding local parametrization of the level set in the unit disc variable $z_1\in\mathbb{D}$.  Now here comes the delicate part of the proof. 

Let us provide some explanation. Our aim is to find an effective superset of this sublevel set, that has nice volumes. The initial idea is that this set is the union of every factor separately becoming small. This is not leading to the desired volume estimates. But if we control the way the factors are becoming small, and modify slightly the initial idea, we can prove that the difference $|\Psi-\alpha|$ becomes small whenever one factor becomes small and the others stay larger. Gluing these sets will lead to the effective superset we are after. For this to be more understandable, we shall separate the proof into steps.

\textbf{Step 1: Definition of the set:}
Define
$$\Omega_{\nu,m,\varepsilon}:=\{|w_2-h^{\alpha,\nu,m}_{\mathbb{H}^2}(w_1)|<\varepsilon|w_1|^{K_\nu}\}.$$
The idea is to prove that if the difference on the left hand side of \ref{disgustingestimate} is small, then $w\in \Omega_{\nu,m,\delta_1}$ for some $\nu,m.$  
 \begin{remark}
The choice of the definition for the regions $\Omega_{\nu,m,\delta_1}$ is intentional and comes from the fact that
\begin{align}
&\frac{C|w_2+q_\nu(w_1)+w_1^{K_{\nu}}\psi^{\nu,m}_{\alpha}(w_1)|}{|w_2+q_\nu(w_1)+w_1^{K_\nu}\psi_{\alpha}^{\nu,m}(w_1)|+|w_1|^{K_\nu}C_{\nu,m}}<\varepsilon\Longrightarrow\notag&&\\ & C|w_2-h^{\alpha,\nu,m}_{\mathbb{H}^2}(w_1)|<\varepsilon(|w_2-h^{\alpha,\nu,m}_{\mathbb{H}^2}(w_1)|+|w_1|^{K_{\nu}}C_{\nu,m})\notag
\end{align}
which for $\varepsilon$ small enough implies $$|w_2-h^{\alpha,\nu,m}_{\mathbb{H}^2}(w_1)|<\frac{\varepsilon}{C-\varepsilon}|w_1|^{K_\nu}\approx\varepsilon|w_1|^{K_{\nu}}.$$
\end{remark}
\textbf{Step 2: Separation of the $\Omega_{\nu,m,\delta_1}.$}

For our argument to work, we need to show that these regions are disjoint from one another. Distinct branches have difference controlled by $\approx|w_1|^{K},$ where $K$ is their order of contact (this comes from the definition of order of contact, see Section \ref{tools} or, again, paper \cite{Clark}).  Hence, for
 $\delta_1<\varepsilon$ sufficiently small, we claim that the regions $\Omega_{\nu,m,\delta_1}$
   are pairwise disjoint.
   \begin{proof} \textbf{(Proof of the claim):}
   To see this, by the order of contact of the branches, we know that $$|h^{\alpha,\nu,m}_{\mathbb{H}^2}(w_1)-h^{\alpha,\ell,m}_{\mathbb{H}^2}(w_1)|\approx |w_1|^{\min\{K_\nu,K_\ell\}}.$$
 Note that for the bound from below, there exists $C>0$ such that
 $$|h^{\alpha,\nu,m}_{\mathbb{H}^2}(w_1)-h^{\alpha,\ell,m}_{\mathbb{H}^2}(w_1)|\ge C |w_1|^{\min\{K_\nu,K_\ell\}}$$
 Letting now $\delta_1$ become smaller, the regions $\Omega_{\nu,m,\delta_1}$ shrink but they stay away from each other, remaining disjoint under control.
 To see this controlled disjointedness, pick two distinct branches $h^{\alpha,\nu_0,m_0}_{\mathbb{H}^2}(w_1)$ and $h^{\alpha,\nu_1,m_1}_{\mathbb{H}^2}(w_1)$ with contact orders $K_{\nu_0},K_{\nu_1}\ge2.$ Assume on the contrary that they overlap. Then, by triangle inequality
 $$|h^{\alpha,\nu_0,m_0}_{\mathbb{H}^2}(w_1)-h^{\alpha,\nu_1,m_1}_{\mathbb{H}^2}(w_1)|\leq|w_2-h^{\alpha,\nu_0,m_0}_{\mathbb{H}^2}(w_1)|+|w_2-h^{\alpha,\nu_1,m_1}_{\mathbb{H}^2}(w_1)|<\varepsilon(|w_1|^{K_{\nu_0}}+|w_1|^{K_{\nu_1}}).$$
 In a small neighborhood around $(0,0),$ specifically for $\varepsilon<C/2,$ one has $|w_1|<1,$ therefore $|w_1|^{K_{\nu_0}}+|w_1|^{K_{\nu_1}}\leq 2|w_1|^{\min\{K_{\nu_0},K_{\nu_1}\}}.$
 This would imply that 
 $$|h^{\alpha,\nu_0,m_0}_{\mathbb{H}^2}(w_1)-h^{\alpha,\nu_1,m_1}_{\mathbb{H}^2}(w_1)|<2\varepsilon|w_1|^{\min\{K_{\nu_0},K_{\nu_1}\}},$$ which is a contradiction, after our choice of $\varepsilon.$ Hence, after shrinking this regions correctly, we keep them disjoint from each other.
 \end{proof}
\textbf{Step 3: Outside of this union the difference $|\Psi-\alpha|$ is bounded from below:}
 After settling our claim, using this controlled disjointedness we want to reach to the following implication: "The whole product of ratios becomes small" then "$w$ belongs to $\bigcup_{\nu,m}\Omega_{\nu,m,\delta_1},$ i.e. outside this union, the other factors will be bounded from below. To see this, assume that $w\not\in\bigcup_{\nu,m}\Omega_{\nu,m,\varepsilon_0},$ for $\varepsilon_0>0.$ Then, for all these indices $\nu,m,$ we would have $$|w_2-h^{\alpha,\nu,m}_{\mathbb{H}^2}(w_1)|\ge \varepsilon_0|w_1|^{K_{\nu}},$$ and by the fact that the function $$f(x)=\frac{x}{x+Ct^K},\quad x\in [0,+\infty)$$ is strictly increasing, we have that 
 $$|\Psi(w_1,w_2)-\alpha|\ge \prod_{\nu=1}^N\prod_{m}^{M_\nu}\frac{|w_2-h^{\alpha,\nu,m}_{\mathbb{H}^2}(w_1)|}{|w_2-h^{\alpha,\nu,m}_{\mathbb{H}^2}(w_1)|+C_{\nu,m}|w_1|^{K_{\nu}}}>\prod_{\nu}^N\prod_{m}^{M_\nu}\frac{\varepsilon_0}{\varepsilon_0+C_{\nu,m}}.$$
 Taking $\delta_1<\prod_{\nu}^N\prod_{m}^{M_\nu}\frac{\varepsilon_0}{\varepsilon_0+C_{\nu,m}}$ leads to a contradiction, therefore reaching to the desired inclusion, which on the bidisc can be written as $$\bigg\{z\in\mathbb{D}^2\cap U:|\phi(z_1,z_2)-\alpha_1|<\delta_1\bigg\} \subset \bigcup_{\nu,m}\bigg\{z\in\mathbb{D}^2\cap U:|z_2-h^{\alpha_1,\nu,m}_{\mathbb{D}^2}(z_1)|<\delta_1|1-z_1|^{K_{\nu}}\bigg\}.$$ 

\textbf{Step 4: The superset for the smooth function:}
Now we pass to the smooth coordinate. Since $\psi \in \mathcal{C}^2(\overline{\mathbb{D}^2})$ and $\psi(1,1)=\tau\in\mathbb{T}$, by transversality assumption, at least one of the partial derivatives $\frac{\partial \psi}{\partial z_1}(1,1)$, $\frac{\partial \psi}{\partial z_2}(1,1)$ is non‑zero. Therefore, we may assume $\frac{\partial \psi}{\partial z_2}(1,1) \neq 0$, otherwise we exchange the variables $z_1$ and $z_2$ and repeat the same argument. Thus, locally near $(1,1)$, by the \textit{implicit function Theorem}, close to $(1,1)$ the level set of $\psi$ can be parametrized by an analytic function $z_2=g(z_1),$ and local transversality means that around a neighborhood of $(1,1),$ $$|h^{\alpha_1,\nu,m}_{\mathbb{D}^2}(z_1)-g(z_1)|\gtrsim |1-z_1|\quad \mathrm{near} \quad 1,\quad \mathrm{branchwise}.$$ Therefore, $$|\psi(z_1,z_2)-\tau|<\delta_2\Longrightarrow|z_2-g(z_1)|\lesssim\delta_2.$$ But
$$|1-z_1|\lesssim|h^{\alpha_1,\nu,m}_{\mathbb{D}^2}(z_1)-g(z_1)|\leq |z_2-h^{\alpha_1,\nu,m}_{\mathbb{D}^2}(z_1)|+|z_2-g(z_1)|\lesssim\delta_1|1-z_1|^{K_\nu}+\delta_2\approx\delta_2,$$ sufficiently close to 1.
To see the last step, shrink once again, and take $|z_1-1|<r,$ for some small $r<\varepsilon.$  Then 
$$|1-z_1|\leq \delta_1r^{K_{\nu}-1}|1-z_1|+\delta_2\Longrightarrow(1-\delta_1r^{K_{\nu}-1})|1-z_1|<\delta_2\Longrightarrow|1-z_1|<\frac{\delta_2}{1-\delta_1r^{K_{\nu}-1}}\approx\delta_2$$
for $\delta_1$ sufficiently small.

\textbf{Step 5: Volume estimates:}
At this point, we can consider the set $$A_{\nu,m}(\delta_1,\delta_2):=\bigcup_{\nu,m}\bigg\{z\in\mathbb{D}^2\cap U:|h_{\mathbb{D}^2}^{\alpha_1,\nu,m}(z_1)-z_2|<\delta_1|1-z_1|^{K_\nu}, |z_1-1|<\delta_2\bigg\}.$$
After our last implication, this set satisfies $$\Phi^{-1}(S(\alpha_1,\tau),\tilde{\delta})\cap U\subset  A_{\nu,m}(\delta_1,\delta_2),$$ therefore
\begin{align} V_{\beta}(\Phi^{-1}(S(\alpha_1,\tau),\tilde{\delta})\cap U))&\leq V_{\beta}\left(A_{\nu,m}(\delta_1,\delta_2)\right)\notag&&\\
&=\int_{ A_{\nu,m}(\delta_1,\delta_2)}(1-|z_1|^2)^\beta(1-|z_2|^2)^{\beta}dV(z_1,z_2)\notag&&\\
&\approx \sum_{\nu,m}\int_{\{z\in\mathbb{D}^2:|1-z_1|<C\delta_2\}}\left(\int_{\{z\in\mathbb{D}^2:|z_2-h^{\alpha_1,\nu,m}_{\mathbb{D}^2}(z_1)|<C\delta_1|1-z_1|^{K\nu}\}}dA_{\beta}(z_2)\right)dA_{\beta}(z_1) \label{line}&&\\
&\approx \sum_{\nu,m}\delta_1^{2}\int_{|1-z_1|<C\delta_2}|1-z_1|^{2K_\nu}(1-|z_1|^2)^\beta
dA(z_1)\notag&&\\
&\approx \sum_{\nu,m}\delta_1^{2}\delta_2^{K_{\nu}(\beta+2)+\beta+2}\notag&&\\
&\lesssim \delta_1^{2}\delta_2^{2\min\{K_\nu\}(\beta+2)+\beta+2}\notag&&\\
&\lesssim \delta_1^{2+\beta}\delta_2^{\beta+2}, \quad \mathrm{for}\quad \beta\in(-1,0]\notag.
\end{align}
The summation is finite in all cases, as we have a finite number of factors.
This suffices to prove boundedness $\beta\in(-1,0].$ The proof works only for $-1<\beta\leq 0$ because our discs in line \ref{line} of our chain, have centers that lay inside $\mathbb{D}$ and not on the boundary. Their volume contribution is $\approx \delta^2$ and not $\delta^{2+\beta}.$ So our proof only works for negative $\beta$ and $\beta=0,$ which corresponds to the classical Bergman space.
\end{proof}
Based on the geometric non-boundedness results and the above sufficiency Theorem, we obtain

\begin{theorem}\textbf{(Transversality is equivalent to boundedness)} Let $\Phi=(\phi,\psi)$ with $\phi=\widetilde{p}/p$ a RIF with a singularity at $\tau\in\mathbb{T}^2,$ and $\alpha_1\in\mathbb{T}$ its generic at $\tau\in\mathbb{T}^2$. Let $\psi\in\mathcal{O}(\mathbb{D}^2,\mathbb{D})\cap \mathcal{C}^2(\overline{\mathbb{D}^2})$ with $\psi(\tau)=\alpha_2\in\mathbb{T}.$ Assume that if the level sets $\mathcal{C}_{\alpha_1}(\phi)$ and $\mathcal{C}_{\alpha_2}(\psi)$ intersect tangentially at $\tau\in\mathbb{T}^2$, then the order of tangential intersection $M$ satisfies $M>2K+1$ where $K=\max\{K_{\phi},K_{\psi}\}.$  Then, the composition operator $C_{\Phi}:A^2_{\beta}(\mathbb{D}^2)\to A^2_{\beta}(\mathbb{D}^2)$ is bounded for all $\beta\in(-1,0],$ if and only if $\mathcal{C}_{\lambda_1}(\phi)$ and $\mathcal{C}_{\lambda_2}(\psi)$ intersect transversally for all $\lambda_1,\lambda_2\in \mathbb{T}.$
\end{theorem}

\begin{proof} Sufficiency has already been proved. For the converse, assume that any tangential intersection on the singularity occurs with order $M>4K+1$ and that the composition operator is bounded. If the intersection was not transversal in smooth or singular points, then we have four possible scenarios. Either a common straight line, a common arc of a curve, a tangential intersection in smooth points, or a tangential intersection in the singularity with order $M>4K+1.$ By Theorems \ref{exceptionalcurve}, \ref{exceptionalline}, \ref{smoothtangentlines} and \ref{singlulartangentlines}, we obtain a contradiction. This proves the Theorem.
\end{proof}

\section{More results and examples}\label{examples} In this section we will present some more secondary results and concrete examples that illustrate the results obtained. In the next Theorem, we essentially present a special case of our main result which holds for all $\beta\ge-1.$ In the first example, we provide a self map of the bidisc induced by the $\kappa-$function and by the function $\phi_{AMY}$ and prove non-boundedness via the techniques of the common line Theorem. In the second example, we consider a case that the common curve Theorem covers. In the third and last example, we consider two self maps which are half smooth-half singular and we showcase the validity of Theorem \ref{characterization1} for exceptional values too.

\subsection{Half smooth-half singular characterization.}
We begin with the non-smooth self map characterization induced by iterations of the function $\kappa(z_1,z_2).$ Note that this Theorem holds for all $\beta\ge-1,$ and the conditions on the partial derivative of the smooth function $\psi$ are equivalent to transversality between the level sets of the RIFs $\{\phi_n\}$ and $\psi.$

\begin{theorem} There exists a family of RIFs $\{\phi_n\}$ of bidegree $(1,n+1)\in\mathbb{N}^2,$ each singular on  $(1,1)\in\mathbb{T}^2$ such that for every $\psi\in\mathcal{O}(\mathbb{D}^2,\mathbb{D})\cap\mathcal{C}^2(\overline{\mathbb{D}^2}),$ with $\psi(1,1)=\tau\in\mathbb{T},$ the composition operator $C_{\Phi_n}:A^2_{\beta}(\mathbb{D}^2)\to A^2_{\beta}(\mathbb{D}^2)$ with $\Phi_n=(\phi_n,\psi)$ is bounded, if and only if both conditions hold:

\begin{itemize}
\item First order condition holds $\forall\zeta\in\mathbb{T}^2\setminus\{(1,1)\}$ such that $\Phi(\zeta)\in\mathbb{T}^2\setminus\{(1,\tau)\}.$
\item  
$\partial_{z_1}\psi(1,1)\neq 0 \quad\mathrm{and}\quad \partial_{z_2}\psi(1,1)\neq0.$
\end{itemize}
\end{theorem}
\begin{proof}
We will construct these symbols with iterations of the $\kappa$ function. Precisely, {\color{refcyan}After multiplying $\psi$ by a unimodular constant, we may assume $\psi(1,1)=1$.}
$$\kappa(z_1,z_2)=\frac{2z_1z_2-z_1-z_2}{2-z_1-z_2}\quad (z_1,z_2)\in\mathbb{D}^2.$$
{\color{refcyan}Set $\phi_0=\kappa^{(0)}=\kappa$. The iterations in this proof are defined by
$$\phi_n(z_1,z_2)=\kappa^{(n)}(z_1,z_2):=\kappa\bigl(z_1,-\kappa^{(n-1)}(z_1,z_2)\bigr),\qquad n\ge1.$$
The minus sign is important in the inclusion used below.}

{\color{refcyan}The first-order condition gives local estimates at the points where $\Phi_n$ is smooth and the condition holds. Near the singularity $(1,1)$, the first coordinate is not smooth, and the preimages of Carleson boxes require a separate estimate.}

We shall prove the Theorem for the function $\kappa(z_1,z_2).$ The general result on the $n-$th iteration $\kappa^{(n)}$ will follow inductively. The estimates follow from the inclusion relation we will prove in the next key Lemma.
\begin{lemma} \label{estimate} 
Let $\kappa(z_1,z_2)$ as before. Then, for all $\beta\ge-1,$
$$V_{\beta}\left(\{z\in\mathbb{D}^2:|\kappa^{(n)}(z_1,z_2)+1|<\delta\}\right)\approx \delta^{2+\beta}, \quad \forall{\delta\in(0,1)}.$$
\end{lemma}
\begin{proof}
We show the estimates for $\kappa.$ The lower estimate can be found in \cite{Me2}, but we will provide the proof here for completeness and for the convenience of the reader. First we prove the estimate from below. We need to estimate the volume of the set $$S_{\delta,\kappa}:=\Bigg\{z\in\mathbb{D}^2:\left|\frac{2(1-z_1)(1-z_2)}{2-z_1-z_2}\right|<C\delta\Bigg\}.$$
Consider the set $$A_{\delta}=\Bigg\{z\in\mathbb{D}^2:|z_1-1|<\delta,|z_2|<1/2\Bigg\}.$$ We observe that the denominator does not vanish on $A_{\delta},$ and, moreover, that if $z\in A_{\delta}$ then also $z\in S_{\delta,\kappa}.$
This implies that $A_{\delta}\subset S_{\delta,\kappa}$ and calculating the volumes we obtain
$$V_{\beta}\left(\Bigg\{z\in\mathbb{D}^2:\left|\frac{2(1-z_1)(1-z_2)}{2-z_1-z_2}\right|<C\delta\Bigg\}\right)>V_{\beta}(A_{\delta})\approx \delta^{2+\beta}.$$
We proceed similarly to the first iteration, after simply observing that
the set
$$S_{\delta,\kappa^{(1)}}:=\Bigg\{z\in\mathbb{D}^2:{\color{refcyan}\left|\frac{2(1-z_1)(1+\kappa(z_1,z_2))}{2-z_1+\kappa(z_1,z_2)}\right|<C\delta}\Bigg\}$$
contains again $A_{\delta},$ yielding the same estimate from below. {\color{refcyan}For a suitable $c_n>0$, the set $\{|z_1-1|<c_n\delta,\ |z_2|<1/2\}$ gives the lower estimate for every $n$ and $\beta>-1$. For $\beta=-1$, the same estimate applies to boundary values on an arc of $\mathbb T$ around $-1$ in the $z_2$ variable; using an arc of length comparable to $\delta$ around $1$ in the $z_1$ variable gives the lower estimate $\gtrsim\delta$.} For the inequality from above, we need to find an effective superset of the sublevel sets.
We claim that 
$$S_{\delta,\kappa}:=\Bigg\{z\in\mathbb{D}^2:\left|\frac{2(1-z_1)(1-z_2)}{2-z_1-z_2}\right|<\delta\Bigg\}\subset\{z\in\mathbb{D}^2:|z_1-1|<\delta\}\bigcup\{z\in \mathbb{D}^2:|z_2-1|<\delta\}$$
To see this, pick points in $S_{\delta,\kappa}.$ We shall assume the contrapositive to reach to the desired conclusion. Assume that it is true that $|z_1-1|\ge \delta$ and $|z_2-1|\ge \delta.$ The defining inequality in $S_{\delta,\kappa}$ can be seen equivalently as
{\color{refcyan}$$|2(1-z_1)(1-z_2)|< \delta|2-z_1-z_2|\leq \delta(|1-z_1|+|1-z_2|).$$}
On the other hand, 
$$|z_1-1||z_2-1|\ge \delta|z_2-1|,$$
and
$$|z_1-1||z_2-1|\ge\delta|z_1-1|.$$
Summing up we receive $$|2(1-z_1)(1-z_2)|\ge \delta(|z_1-1|+|z_2-1|),$$
forcing the defining inequality of $S_{\delta,\kappa}$ impossible, hence
$$S_{\delta,\kappa}\subset\{z\in\mathbb{D}^2:|z_1-1|<\delta\}\bigcup\{z\in \mathbb{D}^2:|z_2-1|<\delta\}.$$ At this point we need to be precise, so the reader is not confused with the subset relation that we stated. Consider the sets $D_{\delta}=\{z\in\mathbb{D}:|z-1|<\delta\}$ and $E_{\delta}=\{z\in\mathbb{D}:|z-1|\ge\delta\}.$
In fact, the set that actually is superset of the sublevel set is explicitly $$S_{\delta,\kappa}\subset (D_{\delta}\times D_{\delta})\bigcup(D_{\delta}\times E_{\delta})\bigcup(E_{\delta}\times D_{\delta}).$$
But this after just calculating the union, coincides with $(D_{\delta}\times\mathbb{D})\cup (\mathbb{D}\times D_{\delta}).$
Now we shall calculate the volume from above. We see that 
\begin{align}
V_{\beta}(\{z\in\mathbb{D}^2:|\kappa(z_1,z_2)+1|<\delta\})&\lesssim V_{\beta}\left(\{z\in \mathbb{D}^2: |z_1-1|<\delta\}\bigcup\{z\in \mathbb{D}^2:|z_2-1|<\delta\}\right)&&\\ &\approx \delta^{2+\beta},
\end{align}
with the implied constant depending only on $\beta.$ {\color{refcyan}For the iteration step, put $A=1-z_1$ and $B=1+\kappa^{(n-1)}(z_1,z_2)$. If both $|A|\ge\delta$ and $|B|\ge\delta$, then $2|AB|\ge\delta(|A|+|B|)\ge\delta|A+B|$. Consequently,
$$\{ |\kappa^{(n)}+1|<\delta\}\subset\{|z_1-1|<\delta\}\cup\{|\kappa^{(n-1)}+1|<\delta\}.$$
Starting from the inclusion for $\kappa$, this proves the same two-strip inclusion for every $n$. In particular, the set $\{|\kappa-1|<\delta\}$ is not used in this induction.} For the first iteration $\kappa^{(1)},$ simply observe that 
\begin{multline}\{z\in\mathbb{D}^2:|z_1-1|<\delta\}\bigcup\{z\in \mathbb{D}^2:|\kappa(z_1,z_2)+1|<\delta\}\\\subset\{z\in \mathbb{D}^2: |z_1-1|<\delta\}\bigcup\{z\in \mathbb{D}^2:|z_2-1|<\delta\} \notag 
\end{multline}
Repeating the iterations leads to the same superset for all $n\in\mathbb{N},$ proving the estimate from above. This proves our crucial lemma.
\end{proof}
At this point we can proceed to the proof of our characterization.
We shall first prove the sufficiency part. We need to show
$$V_{\beta}(\Phi_n^{-1}(S(\zeta,\tilde{\delta})))\leq CV_{\beta}(S(\zeta,\tilde{\delta}))\approx\delta_1^{2+\beta}\delta_2^{2+\beta},\quad \mathrm{for}\quad \mathrm{all}\quad \zeta\in\mathbb{T}^2,\quad \tilde{\delta}=(\delta_1,\delta_2)\in(0,1)^2.$$
{\color{refcyan}We first consider the Carleson box $S((-1,1),\tilde{\delta})$. For this box we have}
{\color{refcyan}$$\Phi^{-1}_n(S((-1,1),\tilde{\delta}))=\bigl\{z\in\mathbb{D}^2:|\kappa^{(n)}(z_1,z_2)+1|<\delta_1,\ |\psi(z_1,z_2)-1|<\delta_2\bigr\}.$$}
This in turn means that 
\begin{multline}
\Phi_n^{-1}(S((-1,1),\tilde{\delta}))\subset\\\left(\{z\in\mathbb{D}^2:|z_1-1|<\delta_1\}\bigcup\{z\in \mathbb{D}^2:|z_2-1|<\delta_1\}\right)\cap\{z\in \mathbb{D}^2:|\psi(z_1,z_2)-1|<\delta_2\}. \notag
\end{multline}
We distribute the union, and we set
$$S_1=\{z\in\mathbb{D}^2:|z_1-1|<\delta_1,|\psi(z_1,z_2)-1|<\delta_2\}$$
and $$S_2=\{z\in\mathbb{D}^2:|z_2-1|<\delta_1,|\psi(z_1,z_2)-1|<\delta_2\}.$$
Now we see that 
$$V_{\beta}(\Phi^{-1}_n(S(-1,1),\tilde{\delta}))\leq V_{\beta}(S_1)+V_{\beta}(S_2),$$
therefore, it suffices for us to prove 
$$V_{\beta}(S_1)\lesssim\delta_1^{\beta+2}\delta_2^{\beta+2} \quad \mathrm{and}\quad V_{\beta}(S_2)\lesssim\delta_1^{\beta+2}\delta_2^{\beta+2}. $$
{\color{refcyan}These estimates follow from the local estimates for the first-order condition, applied to the smooth symbols} $$\Psi_1=(z_1,\psi(z_1,z_2))\quad \mathrm{and}\quad \Psi_2=(z_2,\psi(z_1,z_2)).$$
The derivative matrices for both $\Psi_1,\Psi_2$ are
$$d\Psi_1=\begin{pmatrix}
1 & 0 \\
\partial_{z_1}\psi & \partial_{z_2}\psi&  
\end{pmatrix}
\quad \mathrm{and} \quad d\Psi_2=\begin{pmatrix}
0 & 1 \\
\partial_{z_1}\psi &\partial_{z_2}\psi  
\end{pmatrix}.
$$
{\color{refcyan}At each relevant contact point $(1,\eta)$ for $\Psi_1$, the slice $z\mapsto\psi(1,z)$ is non-constant, since $\partial_{z_2}\psi(1,1)\ne0$. The Julia--Carath\'eodory theorem therefore gives $\partial_{z_2}\psi(1,\eta)\ne0$. Similarly, $\partial_{z_1}\psi(\eta,1)\ne0$ at each relevant contact point for $\Psi_2$. By compactness, the local estimates for the first-order condition can be applied uniformly near these contact sets, and give the volume estimates} $$V_{\beta}(S_1)\lesssim\delta_1^{2+\beta}\delta_2^{2+\beta} \quad \mathrm{and} \quad V_{\beta}(S_2)\lesssim\delta_1^{2+\beta}\delta_2^{2+\beta}.$$
{\color{refcyan}This gives the required upper estimate for the box centered at $(-1,1)$.}
For the converse, note that if the first order conditions fail even at a smooth point of $\Phi_n$ then boundedness fails too. If the first order condition hold on all points besides the singular point of $\kappa^{(n)},$ then we need to check if boundedness fails whenever at least one of the partial derivatives of $\psi$ is zero. Via the volume estimates we have for $|\kappa^{(n)}+1|<\delta,$ it suffices to assume without losing the generality, that $\partial_{z_2}\psi(1,1)=0.$ {\color{refcyan}The other case is handled in the same way using $|z_2-1|<c_n\delta$ and $|z_1|<1/2$: the corresponding lower bound holds, with constants depending on $n$. For $\beta=-1$, the interior disc is replaced by an arc on $\mathbb T$ away from $1$.} Then we will show that the composition operator is not bounded. Consider the slice $g(z)=\psi(1,z),z\in \mathbb{D}.$ The function $g$ is holomorphic on $\mathbb{D}$ and moreover $g(1)=1$ and $g'(1)=0.$ By the Julia-Carath\'eodory Theorem, $g$ has strictly positive angular derivative at the point $1\in\mathbb{T},$ therefore $g$ cannot be non-constant, hence $g(z) \equiv1.$ This in turn implies that $\psi(1,z)\equiv 1.$ Now due to smoothness of $\psi,$ on the set $$A_{\delta}=\{z\in \mathbb{D}^2:|z_1-1|<\delta,|z_2|<1/2\}$$
$$|\psi(z_1,z_2)-1|=|\psi(z_1,z_2)-\psi(1,z_2)|\leq C|z_1-1|,\quad \forall z\in A_{\delta}.$$
By applying Lemma \ref{estimate} we see once again that $A_{\delta}\subset\Phi^{-1}(S((-1,1),C\delta)).$ Calculating the volumes shows non-boundedness of $C_{\Phi},$ implying that the converse is also true, leading to the desired characterization.
\end{proof}
\subsection{Some examples with the functions $\kappa$ and $\phi_{AMY}.$}
Let us at this moment provide a concrete example in which we get non-boundedness of the composition operator for a symbol $\Phi$ induced by the AMY function and $\kappa,$ two of the most well studied RIFs. This is an example to illustrate the \textit{common exceptional line Theorem \ref{exceptionalline}}.
\begin{example} Let $$\Phi=\left(\frac{2z_1z_2-z_1-z_2}{2-z_1-z_2},\frac{4z_1^2z_2-z_1^2-3z_1z_2-z_1+z_2}{4-3z_1-z_2-z_1z_2+z_1^2}\right)\quad (z_1,z_2)\in\mathbb{D}^2.$$ Then $C_{\Phi}:A^2_{\beta}(\mathbb{D}^2)\to A^2_{\beta}(\mathbb{D}^2)$ is not bounded, for all $\beta \ge -1.$
\end{example}
\begin{proof} We borrow some figures from \cite{Clark1} along with their explanatory descriptions. Pay attention to the exceptional value $\alpha=-1$ for both functions. We observe that both level sets contain the same line $z_1=1.$
\begin{figure}[H]
\centering
\includegraphics[width=14cm ]{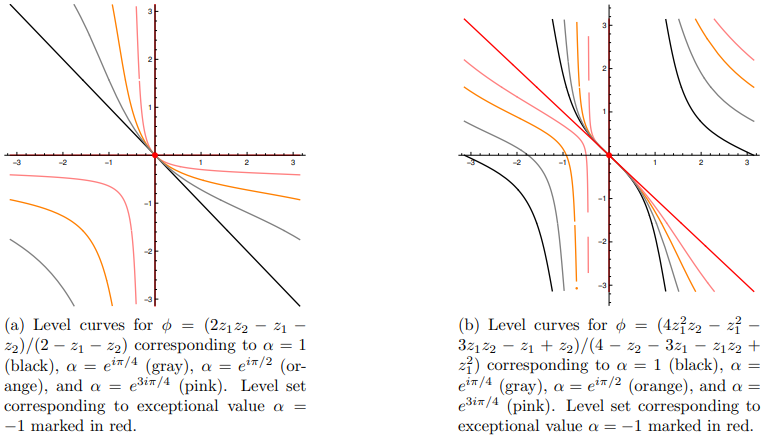}
\end{figure}
Both functions have non-tangential values on their singularity equal to $-1$. Let us first consider the differences $|\kappa(z_1,z_2)+1|$ and $|\phi_{AMY}(z_1,z_2)+1|.$ For the first one, we know $$|\kappa(z_1,z_2)+1|=\left|\frac{2(z_1-1)(z_2-1)}{2-z_1-z_2}\right| \quad (z_1,z_2)\in\mathbb{D}^2.$$ For the AMY function, we know that ${\color{refcyan}\phi_{AMY}(z_1,z_2)=\kappa(z_1,\kappa(z_1,z_2))}.$ Therefore, 
$$|\phi_{AMY}(z_1,z_2)+1|=\left|{\color{refcyan}\frac{2(z_1-1)(1-\kappa(z_1,z_2))}{2-z_1-\kappa(z_1,z_2)}}\right|, \quad (z_1,z_2)\in \mathbb{D}^2.$$
On the set $A_{\delta}$ we defined before, we have that there exist constants $C_1,C_2>0$ such that
$$|\kappa(z_1,z_2)+1|\leq C_1|z_1-1|, \quad |\phi_{AMY}(z_1,z_2)+1|\leq C_2|z_1-1|, \quad \forall(z_1,z_2)\in A_{\delta}.$$
This implies that $$A_\delta\subset \Phi^{-1}(S((-1,-1),C\delta)), $$ for $C=\max\{C_1,C_2\}.$ Calculating the volume we see that
$$V_{\beta}(\Phi^{-1}(S((-1,-1),C\delta)))\gtrsim \delta^{2+\beta}, \quad \forall \delta\in (0,1),$$
preventing boundedness of the operator. 
\end{proof}
Let us provide another example, showcasing the \textit{common curve Theorem \ref{exceptionalcurve}}. The curve here that both RIFs have in common on their level set is the line $\gamma=(e^{i\theta},e^{-i\theta})\subset\mathbb{T}^2$
\begin{example} Let $$\Phi=(\phi_{AMY}(z_1,z_2),\phi_{AMY}(z_2,z_1)),\quad (z_1,z_2)\in\mathbb{D}^2.$$ Then $C_{\Phi}:A^2_{\beta}(\mathbb{D}^2)\to A^2_{\beta}(\mathbb{D}^2)$ is not bounded, for all $\beta\ge-1$.
\end{example}
\begin{proof}
$$|\phi_{AMY}(z_1,z_2)+1|=\left|{\color{refcyan}\frac{2(z_1-1)(1-\kappa(z_1,z_2))}{2-z_1-\kappa(z_1,z_2)}}\right|=\left|4\frac{(1-z_1)(1-z_1z_2)}{p(z_1,z_2)}\right|, \quad (z_1,z_2)\in \mathbb{D}^2,$$
and
$$|\phi_{AMY}(z_2,z_1)+1|=\left|{\color{refcyan}\frac{2(z_2-1)(1-\kappa(z_2,z_1))}{2-z_2-\kappa(z_2,z_1)}}\right|=\left|4\frac{(1-z_2)(1-z_1z_2)}{p(z_2,z_1)}\right|, \quad (z_1,z_2)\in \mathbb{D}^2.$$
Taking the set $A_{\delta}:=\{z\in\mathbb{D}^2: z_1=r\zeta_1, 1-r<\delta,|z_2-\frac{1}{\zeta_1}|<\delta\}$
as in our Theorem, yields volume estimate $\approx \delta^{2\beta+3},$ granting non-boundedness of the corresponding composition operator. 
\end{proof}

The last example is related to Theorem \ref{characterization1}. 

\begin{example} Consider $$\Phi_1=\left(-\frac{2z_1z_2-z_1-z_2}{2-z_1-z_2}, \frac{z_1+2z_2}{3}\right)\quad \mathrm{and}\quad \Phi_2=\left(-\frac{4z_1^2z_2-z_1^2-3z_1z_2-z_1+z_2}{4-3z_1-z_2-z_1z_2+z_1^2},z_1\right).$$ Then $C_{\Phi_1}:A^2_{\beta}(\mathbb{D}^2)\to A^2_{\beta}(\mathbb{D}^2)$ is bounded, while $C_{\Phi_2}:A^2_{\beta}(\mathbb{D}^2)\to A^2_{\beta}(\mathbb{D}^2)$ is not bounded, for all $\beta>-1.$
\end{example}

\begin{proof} Let us begin with $\Phi_1.$ The only relevant points $\zeta\in\mathbb{T}^2$ such that $\Phi(\zeta)\in\mathbb{T}^2$ are the points on the diagonal of $\mathbb{T}^2$, since $\left|\frac{z_1+2z_2}{3}\right|=1$ if and only if $|z_1|=|z_2|=1 $ and $z_1=\lambda z_2$ which holds if and only if $\lambda=1$ Hence, we need to check the first order conditions on the diagonal. For all points on $\Delta\subset\mathbb{T}^2.$ Take $(\eta,\eta)\in\Delta\setminus\{(1,1)\}.$ Then, $$-\kappa(\eta,\eta)=\eta$$ and $\psi(\eta,\eta)=\eta.$ For such points the derivative matrix $d_{\eta}\Phi_1$ satisfies the first order conditions, as $\partial_{z_1}\kappa(\eta,\eta)=\partial_{z_2}\kappa(\eta,\eta)=-\frac{1}{2}$ and $\partial_{z_1}\psi(\eta,\eta)=\frac{1}{3}$ and $\partial_{z_2}\psi=\frac{2}{3}.$
Hence, to finish the proof, one should check if
$$V_{\beta}(\Phi^{-1}(S((1,1),\tilde{\delta})))\lesssim \delta_1^{2+\beta}\delta_2^{2+\beta}.$$
But from Lemma \ref{estimate}, the left hand side of the inequality is bounded by
\begin{multline}
V_{\beta}\Biggl(\Bigg\{\{z\in\mathbb{D}^2:|z_1-1|<\delta_1\}\bigcup\{z\in\mathbb{D}^2:|z_2-1|<\delta_1\}\Bigg\}\bigcap\Bigg\{z\in\mathbb{D}^2:\left|\frac{z_1+2z_2}{3}\right|<\delta_2\Bigg\} \Biggr)\\\lesssim \delta_1^{2+\beta}\delta_2^{2+\beta},
\end{multline}
after a straightforward application of Lemma \ref{estimate}. 

As for the second example, setting $$A_{\delta}=\bigg\{z\in\mathbb{D}^2:|z_1-1|<\delta,|z_2|<1/2\bigg\}$$ we see that $A_{\delta}\subset \Phi^{-1}(S((1,1),\delta))$ and non boundedness follows.

Note here that if the second coordinate was $z_2$ then things might not be the same, as in such a case the $A_{\delta}-$set is not a subset of the reverse image under $\Phi$ of the Carleson box $S((1,1),\delta)$ and we cannot apply the estimate from below as in here.
\end{proof}

\section{Further discussion} There exist plenty of open questions to tackle.\\ 

\textbf{Question 1:} Is Theorem \ref{characterization1} valid for all $\beta>0$ too? \\

\textbf{Question 2:} Is boundedness possible whenever we have tangential intersection of order $M\leq 2K+1$? This question emerges from Theorem \ref{singlulartangentlines} on singular tangential intersection.\\

Our belief is that non boundedness cannot be the expected result, as in the smooth case, if we have tangential order $M=1,$ then our set $A_{\delta}$ has volume bounded from below by $\approx \delta^{2\beta+4},$ which makes the criterion undecided.

\textbf{Question 3:} Is the mixed transversal sufficiency Theorem true for RIF symbols? What about more general self maps with singularities?\\ 

\textbf{Question 4:} All of the non-boundedness results hold for the Hardy space of the bidisc $H^2(\mathbb{D}^2).$ Does our main result hold too?\\

Any of these questions, if answered will lead to an even more thorough understanding of the problem in two variables. We believe that the geometric point of view is quite rich and can provide us with a variety of tools that we can use to generalize even more our results.

\section{Acknowledgements} The author was financially supported by the National
Science Center, Poland, SHENG III, research project 2023/48/Q/ST1/00048.

I would like to thank Alan Sola for the correspondence during the preparation of this work, the detailed comments he provided me with, and the many suggestions on the direction of this project. Moreover, I express my gratitude to Konstantinos Maronikolakis, \L{}ukasz Kosi\'nski and Steph\'ane Charpentier for their comments and offering their time to discuss this note. Moreover, I thank my friend and colleague, Pouriya Torkinejad Ziarati for his valuable insight. Moreover, I would like to thank the anonymous referee for the careful reading of the paper and his/her constructive remarks that enhanced the exposition and the mathematical content of the paper. 
\section*{AI declaration}
No large language model was used in the mathematical work or the writing of this paper prior to its initial submission. Chat GPT Terra 6.1. was used for typesetting and linguistic purposes only for the revision.

\enlargethispage{-24pt}

\end{document}
For this we need to show that these regions are disjoint from one another. Distinct branches have difference controlled by $\approx|w_1|^{K},$ where $K$ is their order of contact (this comes from the definition of order of contact, see Section \ref{tools} or, again, paper \cite{Clark}).  Hence, for
 $\delta_1<\varepsilon$ sufficiently small, we claim that the regions $\Omega_{\nu,m,\delta_1}$
   are pairwise disjoint when we fix $\delta_1.$ 
   \begin{proof} \textbf{(Proof of the claim):}
   To see this, by the order of contact of the branches, we know that $$|h^{\alpha,\nu,m}_{\mathbb{H}^2}(w_1)-h^{\alpha,\ell,m}_{\mathbb{H}^2}(w_1)|\approx |w_1|^{\min\{K_\nu,K_\ell\}}.$$
 Note than for the bound from bellow, there exists $C>0$ such that
 $$|h^{\alpha,\nu,m}_{\mathbb{H}^2}(w_1)-h^{\alpha,\ell,m}_{\mathbb{H}^2}(w_1)|\ge C |w_1|^{\min\{K_\nu,K_\ell\}}$$
 Letting now $\delta_1$ become smaller, the regions $\Omega_{\nu,m,\delta_1}$ shrink but they stay away from each other, remaining disjointed under control.
 To see this controlled disjointedness, pick two distinct branches $h^{\alpha,\nu_0,m_0}_{\mathbb{H}^2}(w_1)$ and $h^{\alpha,\nu_1,m_1}_{\mathbb{H}^2}(w_1)$ with contact orders $K_{\nu_0},K_{\nu_1}\ge2.$ Assume on the contrary that they overlap. Then, by triangle inequality
 $$|h^{\alpha,\nu_0,m_0}_{\mathbb{H}^2}(w_1)-h^{\alpha,\nu_1,m_1}_{\mathbb{H}^2}(w_1)|\leq|w_2-h^{\alpha,\nu_0,m_0}_{\mathbb{H}^2}(w_1)|+|w_2-h^{\alpha,\nu_1,m_1}_{\mathbb{H}^2}(w_1)|<\varepsilon(|w_1|^{K_{\nu_0}}+|w_1|^{K_{\nu_1}}).$$
 In a small neighborhood around $(0,0),$ specifically for $\varepsilon<C/2.$ We have that $|w_1|<1,$ hence $|w_1|^{K_{\nu_0}}+|w_1|^{K_{\nu_1}}\leq 2|w_1|^{\min\{K_{\nu_0},K_{\nu_1}\}}.$
 This would imply that 
 $$|h^{\alpha,\nu_0,m_0}_{\mathbb{H}^2}(w_1)-h^{\alpha,\nu_1,m_1}_{\mathbb{H}^2}(w_1)|<2\varepsilon|w_1|^{\min\{K_{\nu_0},K_{\nu_1}\}},$$ which is a contradiction, after our choice of $\varepsilon.$ Hence, after shrinking this regions correctly, we keep them disjoint from each other.
 \end{proof}